\documentclass[11pt,reqno]{amsart}

\usepackage[colorlinks=true, pdfstartview=FitV, linkcolor=blue, 
citecolor=blue]{hyperref}

\usepackage{a4wide}
\usepackage{amsmath,amssymb} 
\usepackage{bbm}
\usepackage{color}
\usepackage{tikz}
\usetikzlibrary{patterns,snakes}
\usepackage{tkz-graph}
\usepackage{tikz-cd} 
\usepackage{dsfont}
\usepackage{bm}
\usepackage{xcolor}

\theoremstyle{plain}
\newtheorem{theorem}{Theorem}[section]
\newtheorem{prop}[theorem]{Proposition}
\newtheorem{lemma}[theorem]{Lemma}
\newtheorem{coro}[theorem]{Corollary}

\theoremstyle{definition}
\newtheorem{definition}[theorem]{Definition}
\newtheorem{remark}[theorem]{Remark}

\newtheorem{example}[theorem]{Example}

\newcommand{\Z}{{\mathbb Z}}

\newcommand{\R}{{\mathbb R}}
\newcommand{\N}{{\mathbb N}}
\newcommand{\C}{{\mathbb C}}
\newcommand{\mc}{\mathcal}

\newcommand{\card}{\operatorname{card}}
\newcommand{\me}{\mathrm{e}}
\newcommand*\diff{\mathop{}\!\mathrm{d}}

\newcommand{\Mid}{:}

\newcommand{\supp}{\operatorname{supp}}

\tikzset{
  LabelStyle/.style = { rectangle, rounded corners, draw,
                        minimum width = 1em, 
                        font =  },
  VertexStyle/.append style = { inner sep=3pt,
                                font = \large},
  EdgeStyle/.append style = {->, bend left} }

\usepackage{mathtools}
\DeclarePairedDelimiter\abs{\lvert}{\rvert}
\DeclarePairedDelimiter\norm{\lVert}{\rVert}

\usepackage{xcolor}

\begin{document}

\title[Ergodic frequency measures for random substitutions]{
Ergodic frequency measures \\[2mm] for random substitutions
}

\author{Philipp Gohlke}
\address{Fakult\"at f\"ur Mathematik, Universit\"at Bielefeld, \newline
\hspace*{\parindent}Postfach 100131, 33501 Bielefeld, Germany}
\email{pgohlke@math.uni-bielefeld.de}

\author{Timo Spindeler}
\address{Department of Mathematical and Statistical Sciences, \newline
\hspace*{\parindent}632 CAB, University of Alberta, Edmonton, AB, T6G 2G1, Canada}
\email{spindele@ualberta.ca}

\begin{abstract}
We construct a family of ergodic measures on random substitution subshifts (RS-subshifts) associated to a primitive random substitution. In particular, the word frequencies of every finite legal word exist for almost every element of the random substitution subshift with respect to these measures. As an application, we show that for a certain class of random substitutions the measures of maximal entropy are frequency measures.
\end{abstract}



\maketitle

\section{Introduction}\label{SEC:intro}
Deterministic substitutions and the corresponding dynamical systems are well-studied objects in symbolic dynamics, which give rise to applications in other branches of mathematics and science \cite{baake,fogg,queffelec}. 
Early efforts to understand the properties of randomised versions of these substitutions were made in a pioneering work by Peyri\`{e}re \cite{peyriere}. Later, the concept of random substitutions also evoked interest in the physics community \cite{gl,htw}. 
Over the decades, random substitutions have appeared in various disguises and contexts, including $0L$-systems in formal language theory \cite{rozenberg} and expansion-modification systems \cite{Li} that where used to model long-term correlation decay in DNA sequences \cite{salgado,mansilla}. However, a systematic study of random substitutions is still in its infancy. Some basic topological properties were studied in \cite{rs}, a study of periodic points associated to random substitution subshifts was undertaken in \cite{rust}. When passing to the associated subshift, it is desirable to construct measures that keep the structure of the underlying stochastic process meaningful, stressing the point of view that it can be considered as some limiting object of the iterated random substitution. Recently, for a certain class of random substitutions, equilibrium measures have been constructed that are invariant under the substitution action \cite{maldonado}. In this work, we are concerned with measures which are \emph{shift}-invariant, a property which is usually taken as a starting point for assigning objects like diffraction measures or the spectrum of a Schr\"odinger operator to the subshift. In diffraction theory, random substitutions are used to generalise results from the deterministic setting \cite{bss,zaks}. For instance, it is a well-known fact that a point set which comes from certain primitive substitutions gives rise to a pure point diffraction measure, and, consequently, to a pure point dynamical spectrum \cite{bl}. In contrast, in the case of randomised versions of these substitutions, one may obtain a mixed diffraction spectrum and a mixed dynamical spectrum, consisting of a pure point part and an absolutely continuous part \cite{moll2,bss}. The ergodicity of the associated measures plays a vital role to ensure that this diffraction spectrum is the same for almost every realisation. Similarly, the generic choice of a Schr\"odinger spectrum as an almost-surely constant object, relies on the existence of an ergodic measure \cite{damanik}. Therefore, ergodicity is a desirable property of measures that arise from random substitutions. 

We are going to construct these measures in the following way. It is a well-known fact that, if $\rho$ is a primitive substitution, its hull $\mathbb{X}_{\rho}$ is strictly ergodic under the $\Z$-action of the shift \cite[Thm. 4.3]{baake}. Moreover, if $x\in\mathbb{X}_{\rho}$ and $v$ is a subword of $x$, then $\mu([v])$ is given by the word frequency of $v$ in $x$ (which exists uniformly due to Oxtoby's theorem), where $\mu$ is the unique shift invariant and ergodic measure on $\mathbb{X}_{\rho}$ and $[v]$ is the cylinder set which corresponds to $v$. We would like to transfer this idea to random substitutions. Here, we are no longer in the strictly ergodic setting. Instead, we \emph{choose} a measure $\mu$ that is compatible with the word frequencies that arise almost surely in the limit of large inflation words. We establish the ergodicity of $\mu$ by showing that the word frequencies of all legal words exist and coincide for $\mu$-almost all possible realisations. The ergodicity of the frequency measures arising from a certain class of random substitutions was previously claimed in \cite[Thm.~4.22]{moll}. The proof, however, contains a small gap, since it relies on the independence of certain random words, which is only asymptotically true and thus requires some extra work.

The paper is organised as follows. In Section~\ref{sec:rand_subst}, we introduce random substitutions and describe them as Markov processes. Here, we follow the approach presented in \cite{koslicki}, which was also implicitly used in \cite{moll}. The corresponding probability measure is used in Section~\ref{sec:word_freq_exp} to establish well-defined word frequencies in expectation. 
In Section~\ref{sec:word_freq-P-as}, we generalise this result to the almost-sure existence of word frequencies. From these, we construct a shift-invariant measure on the stochastic hull of the random substitution in Section~\ref{sec:main-result}, and show that it is ergodic, comprising the main result of this paper.
In the last section, we compute the metric entropy of a class of frequency measures for an explicit example. Moreover, we show that for a certain class of random substitutions one of the frequency measures is indeed the measure of maximal entropy.

\section{Random Substitutions} \label{sec:rand_subst}

The definition of a random substitution, presented in this section, basically follows the notion of a \emph{substitution Markov chain}, introduced under that name in \cite{koslicki} and going back to \cite{peyriere}. However, our notation is slightly different to parallel the one used for deterministic substitutions \cite{baake}. It is mainly in accordance with \cite{moll}. 

Given a finite alphabet $\mathcal{A}$, a word in $\mathcal{A}$ is any finite concatenation of symbols in $\mathcal{A}$. The length of a word $u$ is denoted by $|u|$. For a given word $u = u_1 \cdots u_n$ and two positions $k,m \in \mathbb{N}$ with $1 \leqslant k \leqslant m \leqslant n$, we define $u_{[k,m]} = u_k \cdots u_m$. The number of occurrences of a word $v= v_1 \cdots v_k$ within a word $u = u_1 \cdots u_n$ is given by $|u|_v = \card \{ 0\leqslant i \leqslant n-k : u_{[i+1,i+k]} = v \}$, which is $0$ in case that $k > n$. If $|u|_v \geqslant 1$, we say that $v$ is a subword of $u$ and we write $v \triangleleft u$ in that situation. Let $\mathcal{A}^{+} = \cup_{n \geqslant 1} \mc A^n$ denote the set of finite words in $\mathcal{A}$. Assuming the alphabet is of size $m = \card \mathcal{A}$, the \emph{Abelianisation} $\Phi$ is a map 
\[
\Phi \colon \mc A^{+} \to \mathbb{R}^m, \quad \Phi(u)_a \, = \, |u|_a, \, \mbox{for all } a \in \mc A.
\]
Given $u,v \in \mc A^{+}$ we write the concatenation of $u$ and $v$ as $uv$. To avoid confusion, we will index words by superscripts, that is, $u = u^1 \cdots u^m$ denotes a concatenation of words $u^1, \ldots, u^m \in \mc A^+$, whereas $u = u_1 \ldots u_n$ is the decomposition of $u$ into letters.
We equip $\mathcal{A}^{+}$ with the discrete topology. Thus, $\mathcal{B} = \mathcal{P}(\mathcal{A}^{+})$, with $\mathcal{P}(\mathcal{A}^{+})$ the power set of $\mathcal{A}^{+}$, gives the $\sigma\text{-algebra}$ of Borel sets.
We denote by $\mathcal{A}^{\mathbb{Z}}$ the set of all bi-infinite sequences in $\mathcal{A}$ and we set $x_{[k,m]} = x_k \cdots x_m$ for all $x \in \mathcal{A}^{\mathbb{Z}}$ and $k\leqslant m \in \mathbb{Z}$. The shift-action on $\mc A^{\mathbb{Z}}$ is defined via $(Sx)_i = x_{i+1}$ for all $x \in \mc A^{\mathbb{Z}}$ and $i \in \mathbb{Z}$.

\begin{definition}
A \emph{random word} on a probability space $(\Omega, \mc F, \mathbb P)$ is a measurable function $\mathcal U \colon (\Omega, \mc F) \to (\mc A^+, \mc P(\mc A^+))$. The \emph{distribution} of $\mc U$ is the probability measure $\mathbb{P} \circ \mc U^{-1}$ on $(\mc A^+, \mc P(\mc A^+))$. We call $u \in \mc A^+$ a \emph{realisation} of $\mc U$ if $\mathbb P [\mc U = u] := \mathbb P (\mc U^{-1}(\{u\}) ) > 0$.
\end{definition}

\begin{remark}
Given a random word $\mc U$ on $(\Omega, \mc F, \mathbb P)$, notions like length or Abelianisation that are defined for words naturally extend to actions on $\mc U$. More precisely, if $f \colon \mc A^+ \to X$, for a measurable space $(X, \mc G)$, we define
\[
f (\mc U) \, := \, f \circ \mc U  \colon \Omega \to X. 
\]
Note that $f(\mc U)$ is measurable without further requirements because we have chosen $\mc P(\mc A^+)$ as the sigma-algebra of words. If $X = \mathbb{R}$, it thus defines a random variable on $(\Omega, \mc F, \mathbb P)$. In particular, $|\mc U|$ is a well-defined random variable on $(\Omega, \mc F, \mathbb P)$, and  $\Phi(\mc U)$ a random vector. Similarly, $\mc U_1$ is a random word (in fact a random letter) defined as $g \circ \mc U$, where $g \colon \mc A^+ \to \mc A, \, u_1 \cdots u_m \mapsto u_1$. 
\end{remark}

\begin{definition}
If $\mc U$ and $\mc V$ are random words on the same probability space $(\Omega, \mc F, \mathbb P)$, we define the \emph{concatenation} of $\mc U$ and $\mc V$ as the random word
\[
\mc U \mc V \colon \Omega \to \mc A^+, \quad (\mc U \mc V)(\omega) \, =\,  \mc U(\omega) \mc V(\omega).
\]
We call $\mc U \mc V$ an independent concatenation if $\mc U$ and $\mc V$ are independent random words, that is $\mathbb{P} \circ (\mc U, \mc V)^{-1} = \mathbb{P} \circ \mc U^{-1} \otimes \mathbb P \circ \mc V^{-1}$ is the product measure on $(\mc A^+)^2$. 
\end{definition}

\begin{remark}
If $\mc U$ and $\mc V$ are defined on different probability spaces $(\Omega_1, \mc F_1, \mathbb P_1)$ and $(\Omega_2, \mc F_2, \mathbb P_2)$ respectively, we can still define an independent concatenation of $\mc U$ and $\mc V$ on the product space $(\Omega_1 \times \Omega_2, \mc F_1 \otimes \mc F_2, \mathbb P_1 \otimes \mathbb P_2)$. This is obtained via $\mc U \mc V (\omega_1, \omega_2) = \mc U(\omega_1) \mc V(\omega_2)$.
\end{remark}

%
%
%
%

The intuitive idea of a random substitution is the following. Given $u \in \mc A^+$, replace each letter in $u$ at random and independently by a word in $\mc A^+$. The distribution for such a replacement shall depend only on the type of the letter that is being replaced. We thus obtain a random word. If we condition on the event that this random word takes a specific value in $\mc A^+$, we can iterate the process. This leads to the structure of a stationary Markov chain with countable state space.
\\We proceed by introducing the data that is needed to define a random substitution. Let us denote by $\mc M_{f}(\mc A^+)$ the space of probability measures on $\mc A^+$ with finite support.

\begin{definition}
A \emph{random substitution rule} on a finite alphabet $\mc A$ is a map 
\[
P^{\vartheta} \colon \mc A \to \mc M_f(\mc A^+), \quad a \mapsto P^{\vartheta}_a.
\] 
The corresponding \emph{transition kernel} is a map $P \colon \mc A^+ \times \mc A^+ \to [0,1]$ with
\[
P(u_1 \cdots u_m, v) \, = \, \sum_{\substack{v^1,\ldots, v^m \in \mc A^+ \\ v^1 \cdots v^m = v}} P^{\vartheta}_{u^{}_1}(v^1) \cdots P^{\vartheta}_{u^{}_m}(v^m).
\]
\end{definition}

In particular, $P(a,v) = P^{\vartheta}_a(v)$ for $a \in \mc A, v \in \mc A^+$ and $P(u,\cdot)$ is a probability measure in $\mc M_f(\mc A^+)$ for all $u \in \mc A^+$. In analogy to transition matrices, powers of the transition kernel $P$ are defined inductively via $P^{k+1}(u,v) = \sum_{w \in \mc A^+} P^k(u,w) P(w,v)$, for all $u,v \in \mc A^+$. 

\begin{remark}
\label{Rem:trans-kernel}
Let us expand a bit on the structure of the transition kernel. Suppose $\mc U_1, \ldots, \mc U_m$ are independent random words with distribution $P^{\vartheta}_{u_1}, \ldots, P^{\vartheta}_{u_m}$, respectively. Then, it is a straightforward calculation to verify that the distribution of $\mc U_1 \cdots \mc U_m$ is given by $P(u_1\cdots u_m, \cdot)$. In other words, $P(u,\cdot)$ describes the distribution if each letter in $u$ is mapped independently to a random word, according to the random substitution rule and the resulting random words are concatenated in the same order as the original letters.
It is worth mentioning that the independence assumption might be relaxed leading to a more general concept of a random substitution, coined \emph{M-system} in \cite{peyriere}. This corresponds to a different choice of the transition kernel.
\end{remark}

Given a random substitution rule $P^{\vartheta}$, we want to define a random substitution $\vartheta$ as a map between random words. There is a technical subtlety in that we need to enlarge the probability space of the original random word in order to make enough `space' for the new one. Suppose $\mc U$ is a random word on $(\Omega_1, \mc F_1, \mathbb P_1)$ and let $(\Omega_2, \mc F_2, \mathbb P_2)$ be a second probability space. With slight abuse of notation we can extend $\mc U$ to a random variable on $(\Omega_1 \times \Omega_2, \mc F_1 \otimes \mc F_2, \mathbb{P}_1 \otimes \mathbb P_2)$, via $\mc U((\omega_1, \omega_2)):= \mc U(\omega_1)$ for all $(\omega_1, \omega_2 ) \in \Omega_1 \times \Omega_2$. Note that $(\mathbb P_1 \otimes \mathbb P_2) \circ \mc U^{-1} = \mathbb P_1 \circ \mc U^{-1}$, so $\mc U$ induces the same distribution on $\mc A^+$ in both cases.

\begin{definition}
\label{Def:random-substitution}
Let $P^{\vartheta}$ be a random substitution rule with transition kernel $P$. A corresponding \emph{random substitution} $\vartheta$ is a map between random words with the following properties. If $\mc U$ is a random word on $(\Omega_1, \mc F_1, \mathbb P_1)$ we define $\vartheta(\mc U)$ to be a random word on $(\Omega, \mc F, \mathbb P) = (\Omega_1 \times \Omega_2, \mc F_1 \otimes \mc F_2, \mathbb P_1 \otimes \mathbb P_2)$ for some probability space $(\Omega_2, \mc F_2, \mathbb P_2)$, satisfying
\begin{equation}
\label{Eq:RS-def}
\mathbb P [\vartheta(\mc U) = v \mid \mc U = u] \, = \, P(u,v),
\end{equation}
for all $u,v \in \mc A^+$. 
\end{definition} 

\begin{remark}
For a fixed $P^{\vartheta}$, there is some freedom to choose the probability space $(\Omega_2, \mc F_2, \mathbb P_2)$ and the specific form of $\vartheta(\mc U)$ as a measurable function on $(\Omega, \mc F, \mathbb P)$ for a given $\mc U$. However, the distribution $\mathbb P \circ (\mc U, \vartheta(\mc U))^{-1}$ on $\mc A^+ \times \mc A^+$ is unambiguously fixed by \eqref{Eq:RS-def}. See the Appendix for an explicit choice of $(\Omega_2, \mc F_2, \mathbb P_2)$ and $\vartheta$, showing that a random substitution always exists.
\end{remark}

Let us discuss some special cases. First, suppose $\Omega_1 = \{a\}, \mathbb P_1 = \delta_a$ and $\mc U(a) = a$ is the trivial random variable equal to $a \in \mc A$. In that case, we write $\vartheta(a) := \vartheta(\mc U)$ and we find 
$\mathbb P [\vartheta(a) = v] = \mathbb P[\vartheta(a) = v \mid \mc U = a ] = P(a,v) = P_a^{\vartheta}(v)$ for all $v \in \mc A^+$. That is, the distribution of $\vartheta(a)$ on $\mc A^+$ is given by $P_a^{\vartheta}$. If $\supp P^{\vartheta}_a = \{v^1,\ldots,v^m\}$ and $P^{\vartheta}_a(v^j) = p_j$ for $j \in \{1,\ldots,m\}$, we represent this distribution in the following form
\[
\vartheta(a) = \begin{cases}
v^1, \quad \mbox{with probability} \; p_1,
\\ \, \vdots
\\ v^m \quad \mbox{with probability} \; p_m,
\end{cases}
\]
compare also the notation in \cite{rs}. This is consistent with the observation $\mathbb P [\vartheta(a) = v^j] = p_j$, for $j \in \{1,\ldots,m\}$, by the above calculation. Note that this holds also in the slightly more general situation that the distribution of $\mc U$ is given by $\delta_a$.

\begin{example}
\label{Ex:r-fib-1}
Let $\mc A = \{a,b \}$. The random Fibonacci substitution rule is given by the two probability distributions $P^{\vartheta}_a$ and $P^{\vartheta}_b$, with $P^{\vartheta}_a(ab) = p_1$, $P^{\vartheta}_a(ba) = p_2 = 1-p_1$ and $P^{\vartheta}_b(a) = 1$. Alternatively, in line with the previous discussion,
\begin{equation}
\label{Eq:r-fib}
\vartheta \colon b \mapsto a, \quad
 a \mapsto \begin{cases}
ab, \quad \mbox{with probability} \, p_1,
\\ ba, \quad \mbox{with probability} \, p_2, 
\end{cases}
\end{equation}
for the distribution of a corresponding random substitution acting on letters. Note that this contains the same information as the random substitution rule. It therefore suffices in order to determine the transition kernel $P$ and hence the common distribution of $(\mc U, \vartheta(\mc U))$ for an arbitrary random word $\mc U$. We call a random substitution with distributions fixed by \eqref{Eq:r-fib} random Fibonacci substitution. Glossing over the ambiguity in the choice of the underlying probability space, we say that \eqref{Eq:r-fib} determines the random substitution $\vartheta$. 
\end{example}

Let us fix an initial random word $\mc U \in \mc A^+$ and construct a sequence of random words $(\vartheta^n(\mc U))_{n \in \N}$ by iterating the random substitution $\vartheta$. 
For $j \in \N$, $\vartheta^j(\mc U)$ is a random word on a probability space $(\bigtimes_{k=1}^{j+1} \Omega_k, \bigotimes_{k=1}^{j+1} \mc F_k, \bigotimes_{k=1}^{j+1} \mathbb P_k)$. We can thus embed all $\vartheta^j(\mc U)$ as functions on a probability space $(\Omega_{\mc U}, \mc F_{\mc U}, \mathbb P_{\mc U}) = (\bigtimes_{k=1}^{\infty} \Omega_k, \bigotimes_{k=1}^{\infty} \mc F_k, \bigotimes_{k=1}^{\infty} \mathbb P_k)$, depending only on the first $j+1$ coordinates. By definition, $(\vartheta^n(\mc U))_{n \in \N}$ forms a stationary Markov chain on $(\Omega_{\mc U}, \mc F_{\mc U}, \mathbb P_{\mc U})$ with  transition probabilities 
\[
\mathbb P_{\mc U} [\vartheta^{n+1}(\mc U) = w \mid \vartheta^n(\mc U) = v]
\, = \, P(v,w),
\]
for all $n \in \N_0$ and $v,w \in \mc A^+$, where we have used $\vartheta^0(\mc U) = \mc U$ as a convention. Inductively, we find $\mathbb P_{\mc U}[\vartheta^{n+k}(\mc U) = w \mid \vartheta^n(\mc U) = v ] = P^k(v,w)$, for all $n,k \in \N$. Since $P^k(v,w) = \mathbb P_v [\vartheta^k(v) = w]$, we obtain
\begin{align*}
\mathbb P_{\mc U} [\vartheta^{n+k}(\mc U) = w]
& \, = \, \sum_{v \in \mc A^+} \mathbb P_{\mc U}[\vartheta^{n+k}(\mc U) = w \mid \vartheta^n(\mc U) = v ] \, \mathbb P_{\mc U} [\vartheta^n(\mc U) = v]
\\ & \, = \, \sum_{v \in \mc A^+} \mathbb P_{\mc U} [\vartheta^n(\mc U) = v] \, \mathbb P_v [\vartheta^k(v) = w].
\end{align*}

%

Next, assume that we start from a deterministic word $\mc U \equiv u$, for some $u = u_1 \ldots u_m \in \mc A^+$. First, observe that $P(u,v) = \bigl( \bigotimes_{j=1}^m \mathbb P_{u_j} \bigr) [\vartheta(u_1) \cdots \vartheta(u_m) = v]$, for all $v \in \mc A^+$, compare Remark~\ref{Rem:trans-kernel}.
Consider the Markov chain $(\vartheta^n(u))_{n \in \N}$ on $(\Omega_u, \mc F_u, \mathbb P_u)$. For each $j \in \{1,\ldots,m\}$ we also have a Markov chain $(\vartheta^n(u_j))_{n \in \N}$. It is possible to realise each $\vartheta^n(u_j)$ as a function on $(\Omega_u, \mc F_u, \mathbb P_u)$ in such a way that $\vartheta^n(u) (\omega) = \vartheta^n(u_1)(\omega) \cdots \vartheta^n(u_m)(\omega)$ while preserving the defining relation $\mathbb P_u [\vartheta^{n+1}(u_j) = w \mid \vartheta^n(u_j) = v] = P(v,w)$ and independence of $\vartheta^n(u_j)$ and $\vartheta^n(u_k)$ for $j \neq k$. We refer to the Appendix for details. In particular,
\[
\mathbb P_u [\vartheta^n(u) = v] \, = \, \mathbb P_u [\vartheta^n(u_1) \cdots \vartheta^n(u_m) = v],
\]
for all $v \in \mc A^+$. This can also be verified using properties of the transition kernel $P$ only.

\begin{remark}
There is some redundancy in notations like $\mathbb P_{\mc U} [\vartheta^n(\mc U) = v]$, as the initial random word $\mc U$ is encoded both as an argument of $\vartheta^n$ and as a subscript of $\mathbb P$. In such situations, we will therefore often write $\mathbb P$ instead of $\mathbb P_{\mc U}$. The same holds for the expectation value $\mathbb E = \mathbb E_{\mc U}$.
\end{remark}

\begin{remark}
\label{Rem:multi-val}
Suppose $\vartheta$ is a random substitution on the alphabet $\mc A$. Let us define a set-valued map on $\mc A$ via $\hat{\vartheta}(a) := \supp (\mathbb P \circ \vartheta(a)^{-1}) \subset \mc A^+$. In words, $\hat{\vartheta}(a)$ consists of the set of possible realisations of $\vartheta(a)$. We can extend $\hat{\vartheta}$ to words via 
\[
\hat{\vartheta}(u_1 \cdots u_m) \, = \, \{ v^1 \cdots v^m \mid v^j \in \hat{\vartheta}(u_j) \; \mbox{for all} \; j \in \{1,\ldots,m\} \}
\]
and to sets of words as $\hat{\vartheta}(A) = \cup_{u \in A} \hat{\vartheta}(u)$. This is what is called a `random substitution' in \cite{grs}. It is a straightforward exercise to verify inductively that $\hat{\vartheta}^n(u) = \supp (\mathbb P \circ \vartheta^n(u)^{-1})$ for all $n \in \N$, so $v \in \hat{\vartheta}^n(u)$ if and only if $\mathbb{P} [\vartheta^n(u) = v] > 0$. We refer to $\hat{\vartheta}$ as a \emph{multi-valued substitution} in our present context.
\end{remark}

\begin{example}
Recall the random Fibonacci substitution $\vartheta$ from Example~\ref{Ex:r-fib-1}.
The distribution of\/ $\vartheta^2(b) = \vartheta(\vartheta(b)) = \vartheta(a)$ is immediate. The probability that\/ $\vartheta^2(a) = aba$ can be determined to be
\begin{align*}
\mathbb{P} [\vartheta^2(a) = aba]
& \, = \, \sum_{v \in \mc A^+} \mathbb P [\vartheta(a) = v ] \mathbb P[\vartheta(v) = aba]
\\& \, = \, \mathbb{P} [\vartheta(a) = ab] \, \mathbb{P} [\vartheta(a)\vartheta(b) = aba] + \mathbb{P} [\vartheta(a) = ba] \, \mathbb{P} [\vartheta(b)\vartheta(a) = aba]
\\&\, =\, p_1^2 + p_2^2.
\end{align*}
Note that there are two different realisations $(ab, aba)$ and $(ba, aba)$ of $(\vartheta(a),\vartheta^2(a))$, contributing to the event $[\vartheta^2(a) = aba]$. 
The full distribution of\/ $\vartheta^2(a)$ is given by
\[
\vartheta^2(a) \, = \, \begin{cases}
aab, \quad \mbox{with probability } p_2 p_1,
\\aba, \quad \mbox{with probability } p_1^2 + p_2^2,
\\baa, \quad \mbox{with probability } p_1 p_2 .
\end{cases}
\]
\end{example} 

\begin{example}
Consider the random substitution
\[
\vartheta \colon a \mapsto \begin{cases}
b, \quad \mbox{with probability } p_1,
\\ba, \quad \mbox{with probability } p_2,
\end{cases}
\quad
b \mapsto \begin{cases}
b, \quad \mbox{with probability } q_1,
\\ab, \quad \mbox{with probability } q_2.
\end{cases}
\]
Then, the event\/ $[\vartheta(ab) = bab]$ can be realised either via\/ $\vartheta(a) = ba$, $\vartheta(b) = b$ or via\/ $\vartheta(a) = b$, $\vartheta(b) = ab$. For the transition kernel\/ $P$, we therefore find\/ 
\begin{align*}
P(ab,bab) & \, = \, P^{\vartheta}_{a}(ba)P^{\vartheta}_b(b) + P^{\vartheta}_a(b) P^{\vartheta}_b(ab) 
\\ & \, = \, \mathbb P[\vartheta(a) = ba] \, \mathbb P[\vartheta(b) = b] + 
\mathbb P [\vartheta(a) = b] \, \mathbb P[\vartheta(b) = ab]
 \, = \, p_2 q_1 + p_1 q_2.
\end{align*}
This is one of the cases where the pair $(\vartheta(a),\vartheta(b))$ contains more information than their concatenation $\vartheta(a) \vartheta(b) = \vartheta(ab)$.
\end{example}

In the same spirit as for deterministic substitutions, we can define a substitution matrix which captures some information about the Abelianisation of the random substitution process.

\begin{definition}
\label{Def:mean-matrix}
The \emph{mean substitution matrix}\/ $M_{\vartheta}$ of a random substitution\/ $\vartheta$ is defined via\/ $(M_{\vartheta})_{ab} = \mathbb{E}[\abs{\vartheta(b)}_{a}] = \mathbb{E}[ \Phi(\vartheta(b))_a]$ for all\/ $a,b \in \mc A$. A random substitution\/ $\vartheta$ is called \emph{primitive} if\/ $M_{\vartheta}$ is a primitive matrix. In this case, it is called \emph{expanding} if 
\/ $\mathbb{P} [\abs{\vartheta(a)} > 1]>0$ for some\/ $a \in \mathcal{A}$. For a primitive random substitution, we denote by\/ $\lambda$ the Perron--Frobenius (PF) eigenvalue of\/ $M_{\vartheta}$ and by\/ $R$ and\/ $L$ the corresponding right and left eigenvectors, respectively. The normalisation of these vectors is chosen such that\/ $\norm{R}_1 = 1$ and\/ $L \cdot R = 1$.
\end{definition}

\begin{remark}
Unlike in the deterministic setting, a primitive random substitution need not be expanding. An easy counterexample is given by
\[
\vartheta \colon\ a,b \mapsto \begin{cases}
a, \quad \mbox{with probability } p_1,
\\ b, \quad \mbox{with probability } p_2 ,
\end{cases}
\quad\ \ M_{\vartheta} \, = \, \begin{pmatrix}
p_1 & p_1
\\ p_2 & p_2
\end{pmatrix}.
\]
\end{remark}

We can construct from a random substitution the set of legal words $\mathcal{L}_{\vartheta}$ and the symbolic hull $\mathbb{X}_{\vartheta}$, much as in the deterministic case, compare \cite{moll}, \cite{rs}.
\begin{definition}
The \emph{language} of a random substitution\/ $\vartheta$ on\/ $\mathcal{A}$ is defined as 
\[
\mathcal{L}_{\vartheta} \, = \, \{ u \in \mathcal{A}^{+} \Mid u \triangleleft v, \, \mathbb{P}[\vartheta^n(a) = v]>0, \, \text{for some} \; n \in \mathbb{N}, a \in \mathcal{A} \}.
\] 
The \emph{set of ($\vartheta$-)legal words of length\/ $\ell$} is given by\/ $\mathcal{L}^{\ell}_{\vartheta} = \{ u \in \mathcal{L}_{\vartheta} \Mid \abs{u} = \ell \}$.
\end{definition}

The definition of $\mathcal{L}_{\vartheta}$ is independent of the probability vectors  (as long as none of their entries is $0$), making it a purely combinatorial object. In fact, it can already be constructed from the multi-valued substitution $\hat{\vartheta}$ pertaining to $\vartheta$, compare Remark~\ref{Rem:multi-val}. The same holds for the subshift that we can construct from this language. 
\begin{definition}
For a random substitution $\vartheta$ on $\mathcal{A}$, we define the \emph{random substitution subshift} (RS-subshift) or \emph{stochastic hull} as
\[
\mathbb{X}_{\vartheta} \, = \, \{ x \in \mathcal{A}^{\mathbb{Z}} \Mid x_{[k,m]} \in \mathcal{L}_{\vartheta}  \text{ for all }  k \leqslant m \}.
\]
\end{definition}
A probabilistic structure on $\mathbb X_{\vartheta}$ will be obtained by a convenient choice of a measure on this subshift in Section~\ref{sec:main-result}.

\section{word frequencies, convergence in expectation}
\label{sec:word_freq_exp}
Many of the results in this section essentially go back to \cite{koslicki}. However, we do not assume the existence of a letter $b \in \mathcal{A}$ and a realisation $\omega\in \Omega_b$ such that $\vartheta(b)(\omega)$ starts with $b$, which is taken as a standing assumption in \cite{koslicki}. Also note that the terminology in \cite{koslicki} is slightly different in so far as  properties that hold under taking expectation are often called `almost sure'.

An important tool for determining letter frequencies of a deterministic substitution is its corresponding induced substitution \cite[Ch.~5.4.1]{queffelec}. Before we continue, we need to extend this notion to the case of expanding random substitutions. For an arbitrary $\ell \in \mathbb{N}$, let $\mathcal{A}_{\ell} := \mathcal{L}^{\ell}_{\vartheta}$ be the alphabet consisting of letters that are $\vartheta$-legal words of length $\ell$. We define a subset of consistent right collared words as
\[ 
\mc D_{\ell} \, = \, \{ V \in \mc A_{\ell}^+ \Mid V = (v_{[k,k+\ell-1]})_{1 \leqslant k \leqslant m} \; \mbox{for some} \, m \in \N, v \in \mc A^+ \}.
\]
For $v \in \mc A^+$, let us assume that $\vartheta^n(v_j)$ is a random variable on $(\Omega_v, \mc F_v, \mathbb P_v)$ for all $1 \leqslant j \leqslant |v|$, compare the Appendix. In that case, we set $\vartheta^n(v_k \cdots v_m)(\omega) = \vartheta^n(v_k)(\omega) \cdots \vartheta^n(v_m)(\omega)$ for $1 \leqslant k \leqslant m \leqslant |v|$ and $\omega \in \Omega_v$.

\begin{definition}
Let $\vartheta$ be an expanding random substitution, $\ell \in \N$ and $V = V_1 \cdots V_m = (v_{[k,k+\ell-1]})_{1 \leqslant k \leqslant m} \in \mc D_{\ell}$ with $v = v_1 \cdots v_{m+\ell -1}$. For $n \in \N$, the \emph{($\ell-$)induced substitution} of $\vartheta^n(v)$ is a random word $\vartheta^n_{\ell}(V)$ on the probability space $(\Omega_v, \mc F_v, \mathbb P_v)$, with values in $\mc D_{\ell}$. It is defined as
\[
\vartheta_{\ell}^n(V)(\omega) \, = \, \left( \vartheta^n(v)(\omega)_{[k,k+\ell-1]}  \right)_{1 \leqslant k \leqslant \abs{\vartheta^n(v^{}_1 \cdots v^{}_m)(\omega)}}.
\]
\end{definition}

The pairs $(\vartheta^n(v),\vartheta^n(v_1 \cdots v_m))_{n \in \N}$ form a stationary Markov chain by definition. Since each $\vartheta^n_{\ell}(V)$ is a function of $(\vartheta^n(v),\vartheta^n(v_1 \cdots v_m))$, the same holds for the sequence $(\vartheta^n_{\ell}(V))_{n \in \N}$. 
We have the identity $\vartheta^n_{\ell}(V) = \vartheta_{\ell}^n(V_1) \cdots \vartheta^n_{\ell}(V_m)$, with
\[
\vartheta_{\ell}^n(V_j)(\omega) \, = \, \left( \vartheta^n(V_j)(\omega)_{[k,k+\ell-1]}  \right)_{1 \leqslant k \leqslant \abs{\vartheta^n(v^{}_j)(\omega)}}
\]
for all $n \in \N$, $j \in \{1,\ldots,m\}$ and $\omega \in \Omega_u$.

\begin{remark}
It is possible to define an induced substitution $\vartheta_{\ell}$ also as a map between $\mc D_{\ell}$-valued random words, similar to Definition~\ref{Def:random-substitution}. However, the transition kernel $P_{\ell}$ corresponding to $\vartheta_{\ell}$ does not have the same product structure as $P$. This corresponds to the fact that for $V= V_1 \cdots V_m$, the random words $\vartheta_{\ell}(V_1),\ldots, \vartheta_{\ell}(V_m)$ are not independent. That is because e.g.\ the words $V_1$ and $V_2$ have some overlap which, by force, is mapped to the same word under $\vartheta$ for any realisation. The Markov chain $(\vartheta^n_{\ell}(V))_{n \in \N}$ forms an M-system in the sense of Peyri\`{e}re \cite{peyriere}.
\end{remark}

We again define an Abelianisation map $\Phi_{\ell} \colon \mc D_{\ell} \to \N_0^{\card \mc A_{\ell}}$ via $\Phi_{\ell}(V)_u = |V|_u$ for all $V \in \mc D_{\ell}$ and $u \in \mc A_{\ell}$. The mean substitution matrix $M_{\vartheta_{\ell}}$ is given by $(M_{\vartheta_{\ell}})_{uv} = \mathbb E [|\vartheta_{\ell}(v)|_u]= \mathbb E \Phi_{\ell}(\vartheta_{\ell}(v))_u $ for all $u,v \in \mc A_{\ell}$. Denote by $L^{(\ell)}$ and $R^{(\ell)}$ the corresponding left and right eigenvectors, normalized as $L^{(\ell)} \cdot R^{(\ell)} =1$ and $\norm{R^{(\ell)}}_1 = 1$. Note that for $\ell =1$, we have $\Phi_{\ell} = \Phi$ and $M_{\vartheta_{\ell}} = M_{\vartheta}$, compare Definition~\ref{Def:mean-matrix}.

\begin{prop}
\label{1_prop_diagram}
Let\/ $\vartheta$ be a random substitution,\/ $\ell \in \N$ and\/ $V = V_1 \cdots V_m = (v_{[k,k+\ell-1]})_{1 \leqslant k \leqslant m}$ with\/ $v = v_1 \cdots v_{m+\ell -1} \in \mc A^+$. If\/ $\ell > 1$, suppose further that\/ $\vartheta$ is expanding. Then, for all\/ $n \in \N$, we have the following conditional expectation,
\[
\mathbb E[\Phi_{\ell}(\vartheta^n_{\ell}(V)) \, | \, \Phi_{\ell}(\vartheta^{n-1}_{\ell}(V))] \, = \, M_{\vartheta_{\ell}} \Phi_{\ell}(\vartheta^{n-1}_{\ell}(V)).
\]
In particular,\/ $\mathbb E[\Phi_{\ell}(\vartheta^n_{\ell}(V)) ]= M_{\vartheta_{\ell}} \mathbb E[\Phi_{\ell}(\vartheta^{n-1}_{\ell}(V))] = M^n_{\vartheta_{\ell}} \Phi_{\ell}(V)$.
\end{prop}

\begin{proof}
Because $(\vartheta^n_{\ell}(V))_{n \in \N}$ is a stationary process, it suffices to show 
\[
\mathbb E[\Phi_{\ell}(\vartheta_{\ell}(V)) ] 
\, = \, \mathbb E[\Phi_{\ell}(\vartheta_{\ell}(V)) | \Phi_{\ell}(V) ] \, = \, M_{\vartheta_{\ell}} \Phi_{\ell}(V). 
\]
For $u \in \mc A_{\ell}$, we can conclude from the pointwise identity 
$
\abs{\vartheta_{\ell}(V)(\omega)}_u \, = \, \sum_{i=1}^m \abs{\vartheta_{\ell}(V_i)(\omega)}_u
$
the corresponding relation under taking expectations,
\begin{equation*}
\mathbb{E} \abs{\vartheta^{}_{\ell}(V)}_u \, = \, \sum_{i=1}^m \mathbb{E} \abs{\vartheta^{}_{\ell}(V_i)}_u \, = \, \sum_{v \in \mc A_{\ell}} (M_{\vartheta_{\ell}})_{uv} \Phi_{\ell}(V)_v 
\, = \, (M_{\vartheta_{\ell}} \Phi_{\ell}(V))_u.
\end{equation*}
Since $\mathbb E \Phi_{\ell}(\vartheta_{\ell}(V))_u = \mathbb E |\vartheta_{\ell}(V)|_u$, this proves the claim.
\end{proof}

For $\ell = 1$, the induced substitution $\vartheta_1$ coincides with the original substitution $\vartheta$. This yields the following alternative characterisation of primitive substitutions, compare \cite[Rem. 6]{rs} and \cite[Rem. 2.13]{moll}.

\begin{coro} \label{1-coro-prim-subst}
Let\/ $\vartheta$ be a random substitution rule on\/ $\mathcal{A}$. Then,\/ $\vartheta$ is primitive if and only if there exists a number\/ $k \in \mathbb{N}$ such that, for all\/ $a,b \in \mathcal{A}$, we have\/ $\mathbb{P} [a \triangleleft \vartheta^k(b)]>0$.
\end{coro}
\begin{proof}
From Proposition~\ref{1_prop_diagram}, we know $\mathbb{E}\Phi(\vartheta^k(b))_a =  (M^k_{\vartheta}\Phi(b))_a = (M^k_{\vartheta})_{ab}$. Therefore,
\[
(M^k_{\vartheta})_{ab} \, > \, 0 \, \iff \, \mathbb{E}\Phi(\vartheta^k(b))_a \, > \, 0 \, \iff \, 
\mathbb{P} [\abs{\vartheta^k(b)}_a > 0] \, > \, 0 .
\]
\end{proof}

It was shown in \cite[Prop. 26]{rs} that $M_{\vartheta^{}_{\ell}}$ is primitive if $M_{\vartheta}$ is primitive and in \cite[Lem. 2.4.9]{koslicki} that the PF eigenvalues of both matrices coincide.
The following result is inspired by \cite[Lem. 2.4.5, Lem. 2.4.6, Cor. 2.4.8]{koslicki}. Note that our results differs from the one given in the proof of \cite[Lem. 2.4.5]{koslicki} by a factor of $L_u^{(\ell)}$ as we fixed the norm of the right eigenvector to be $1$. 

\begin{lemma} \label{1-lemma-exp-convergence}
Let\/ $\vartheta$ be an expanding primitive random substitution on the finite alphabet\/ $\mathcal{A}$. For any\/ $\ell \in \mathbb{N}$, let\/ $\vartheta_{\ell}$ be the induced substitution with mean substitution matrix\/ $M_{\vartheta_{\ell}}$, PF eigenvalue\/ $\lambda$ and right and left PF-eigenvector\/ $R^{(\ell)}$ and\/ $L^{(\ell)}$, normalised as above. Then, we have the following convergence results: 
\begin{align}
& \lim_{n \rightarrow \infty} \frac{\mathbb{E}\Phi_{\ell}(\vartheta_{\ell}^n(u))}{\lambda^n} \, = \, L^{(\ell)}_u R^{(\ell)},  \label{eq_expec_conv}
\\[5pt] & \lim_{n \rightarrow \infty} \frac{\mathbb{E}\abs{\vartheta^{n+1}_{\ell}(u)}}{\mathbb{E}\abs{\vartheta^n_{\ell}(u)}} \, = \, \lambda, 
\\[5pt] & \lim_{n \rightarrow \infty} \frac{\mathbb{E}\Phi_{\ell}(\vartheta^n_{\ell}(u))}{\mathbb{E}\abs{\vartheta^n_{\ell}(u)}} \, = \, R^{(\ell)},
\end{align}
for all\/ $u \in \mathcal{A}_{\ell}$. The same statement is still true for non-expanding\/ $\vartheta$ if\/ $\ell =1$.
\end{lemma}
\begin{proof}
Note that $\norm{\Phi_{\ell}(\vartheta^n_{\ell}(u))}_1 = \abs{\vartheta^n_{\ell}(u)}$ and the expectation value is a linear operator, so the second and the third relations follow immediately from the first, observing the identity $\frac{1}{\lambda^n} \mathbb{E} \abs{\vartheta^n_{\ell}(u)} \xrightarrow{n \to \infty } L^{(\ell)}_u \neq 0$. It thus remains to show Eq.~\eqref{eq_expec_conv}. 

Using Proposition~\ref{1_prop_diagram}, we find $\mathbb{E}\Phi_{\ell}(\vartheta_{\ell}^n(u)) = M_{\vartheta^{}_{\ell}}^n \Phi_{\ell}(u)$ and thus, by an application of the PF theorem, 
\[
\frac{\mathbb{E}\Phi_{\ell}(\vartheta_{\ell}^n(u))}{\lambda^n}
\, = \, \frac{1}{\lambda^n} M_{\vartheta_{\ell}}^n \Phi_{\ell}(u) 
\, \xrightarrow{\, n \to \infty \,} \, R^{(\ell)} (L^{(\ell)} \cdot e_u) 
\, = \, R^{(\ell)} L^{(\ell)}_u.
\]
\end{proof}

In the next result, we establish an exponentially decaying bound for the probability that the length of $\vartheta^n(a)$ grows sublinearly. This will be a convenient tool in later constructions.
\begin{prop} \label{length-as-infty}
If\/ $\vartheta$ is an expanding primitive substitution, there are constants\/ $C,K > 0$ such that, for all\/ $a \in \mc A$ and large enough\/ $n$,
\[
\mathbb{P} \left[ | \vartheta^n(a) | < K n \right]
\, \leqslant \, \me^{- C n}.
\]
In particular,
\[
\lim_{n \rightarrow \infty} \abs{\vartheta^n(a)} \, = \, \infty,
\]
almost surely for all\/ $a \in \mathcal{A}$.
\end{prop}
\begin{proof}
The second claim is an immediate consequence of the first one by an application of the Borel--Cantelli lemma.

Let $ b \in A$ and $u$ be a word such that $|u|\geqslant 2$ and $\mathbb P[\vartheta(b)(\omega)=u] > 0$. From the primitivity of $\vartheta$, we deduce that there is $\widetilde N\in\N$ such that, for all $a \in \mc A$, we have $b \triangleleft \vartheta^{\widetilde N}(a)$, and consequently $u\triangleleft \vartheta^{\widetilde N+1}(a)$ with positive probability $\widetilde{q}$. Set $N:=\widetilde N+1$. It follows 
\[
\mathbb P\big[|\vartheta^N(a)| \geqslant 2\big] \, \geqslant \, \widetilde q,
\]
hence,
\begin{equation*} 
\mathbb P\big[|\vartheta^N(a)| = 1\big] 
\, \leqslant \, (1-\widetilde q) \, =: \, q \, < \,1.
\end{equation*}

Since this holds for all $a \in \mc A$, we obtain for every $v \in \mc A^+$, $\mathbb P[|\vartheta^N(v)| = |v|] \leqslant q^{|v|} \leqslant q$. By stationarity, this yields for every $k \in \N$,
\begin{equation}
\label{Eq:no-growth}
\mathbb P[|\vartheta^{N k}(a)| \, = \, |\vartheta^{N(k-1)}(a)|] \leqslant q.
\end{equation}
For every $n \in \N$ and $\omega \in \Omega_a$, let us define
\[
J_n(\omega) \, = \, \{j \in \{1,\ldots,n\} \Mid |\vartheta^{Nj}(a)(\omega)| = |\vartheta^{N(j-1)}(a)(\omega)| \}.
\]
Combining Eq.~\eqref{Eq:no-growth} with the Markov property, we obtain for every $J \subset \{1,\ldots,n\}$,
\[
\mathbb P[J \subset J_n] \, \leqslant \, q^{\card(J)}.
\]
Suppose $0<K<1$. Then, $|\vartheta^{Nn}(a)(\omega)| \leqslant K n$ implies $\card (J_n(\omega)) \geqslant n - \lceil Kn \rceil.$
Simple combinatorial arguments yield
\begin{align*}
\mathbb{P} \left[ | \vartheta^{Nn}(a) |  \leqslant Kn \right] 
& \, \leqslant \, \mathbb{P} \left[ \card J_n \geqslant n - \lceil Kn \rceil \right] 
\, = \, \sum_{k=n-\lceil Kn \rceil}^{n} \mathbb{P} \left[ \card J_n = k \right] 
\\ & \, \leqslant \, \sum_{k=n-\lceil Kn \rceil}^{n} {n \choose k} q^k
 \, \leqslant \, \left( \lceil Kn \rceil +1 \right) {n \choose \lceil Kn\rceil } q^{n - \lceil Kn \rceil}.
\end{align*}
%

Using Stirling's formula, we find
\[
\liminf_{n \rightarrow \infty} \frac{1}{n} \log {n \choose \lceil Kn\rceil } 
\, = \, - K \log(K) - (1-K) \log(1-K)
\]
and thus
\begin{align*}
\liminf_{n \rightarrow \infty} \frac{1}{n} \log\left(\mathbb{P} \left[ | \vartheta^{Nn}(a) |  \leqslant Kn \right] \right) 
& \, \leqslant \, (1 -K) \log(q) - K \log(K) - (1-K) \log(1-K) 
\\ & \,\xrightarrow{K \rightarrow 0} \, \log(q) \, < \, 0.
\end{align*}
Thereby, we can find $\widetilde{K}, \widetilde{C} > 0$ such that 
\[
\mathbb{P} \left[ | \vartheta^{Nn}(a) |  \leqslant \widetilde{K} n \right] \, \leqslant \, \me^{- \widetilde{C} n},
\]
for large enough $n$.
Setting $C = \frac{\widetilde{C}}{2N}$ and $K = \frac{\widetilde{K}}{2N}$, the assertion follows by a straightforward interpolation argument.
\end{proof}

\begin{coro}
\label{coro:expanding-lambda}
A primitive random substitution\/ $\vartheta$ is expanding if and only if the corresponding PF-eigenvalue\/ $\lambda$ fulfils\/ $\lambda >1$. 
\end{coro}
\begin{proof}
We will use the assertion of Lemma~\ref{1-lemma-exp-convergence} for the case $\ell = 1$, such that it is still true for non-expanding substitutions. Therefore, $\lambda > 1$ is equivalent to $\mathbb{E}\abs{\vartheta^n(a)} \rightarrow \infty$ as $n \rightarrow \infty$. So, $\lambda >1$ clearly implies that $\vartheta$ is an expanding substitution. On the other hand, if $\vartheta$ is expanding, $\abs{\vartheta^n(a)(\omega)} \rightarrow \infty$, as $n \rightarrow \infty$, for all $\omega \in B$ with $\mathbb{P}(B) = 1$, by Proposition~\ref{length-as-infty}. Therefore,
\[
\liminf_{n \rightarrow \infty} \mathbb{E} \abs{\vartheta^n(a)} 
\, = \, \liminf_{n \rightarrow \infty} \int_B \abs{\vartheta^n(a)(\omega)} \diff \mathbb{P} ( \omega) 
\, \geqslant \, \int_B \liminf_{n \rightarrow \infty} \abs{\vartheta^n(a)(\omega)} \diff \mathbb{P} (\omega) 
\, = \, \infty,
\]
by Fatou's Lemma. Thus, $\lim_{n \rightarrow \infty} \mathbb{E} \abs{\vartheta^n(a)} = \infty$ for all $a \in \mathcal{A}$, implying $\lambda > 1$.
\end{proof}

We finally conclude with the following result --- compare the proof of \cite[Thm. 2.4.10]{koslicki}. 
\begin{prop}\label{thm-expec-conv}
Let\/ $\vartheta$ be an expanding, primitive random substitution on\/ $\mathcal{A}$. For any\/ $\ell \in \mathbb{N}$, let\/ $\vartheta_{\ell}$ be the induced substitution with mean substitution matrix\/ $M_{\vartheta^{}_{\ell}}$ and right normalised PF eigenvector\/ $R^{(\ell)}$. Then, for any\/ $v \in \mathcal{L}^{\ell}_{\vartheta}$ and\/ $a \in \mathcal{A}$,
\[
\lim_{n \rightarrow \infty} \frac{\mathbb{E}\abs{\vartheta^n(a)}_v}{\mathbb{E}\abs{\vartheta^n(a)}} \, = \, R^{(\ell)}_v.
\]
Furthermore, for any\/ $u \in \mc L_{\vartheta}^{\ell}$ with\/ $u_1 = a$,
\[
L_u^{(\ell)} \, = \, \lim_{n \rightarrow \infty} \frac{\mathbb E \abs{\vartheta_{\ell}^n(u)}}{\lambda^n} 
\, = \, \lim_{n \rightarrow \infty} \frac{\mathbb E \abs{\vartheta^n(a)}}{\lambda^n} 
\, = \, L_a .
\]
\end{prop}
\begin{proof}
Let $\ell \in \mathbb{N}$, $a \in \mathcal{A}$ and $v \in \mathcal{L}^{\ell}_{\vartheta}$. Since $\vartheta$ is expanding and primitive, there are words of arbitrary length containing $a$ --- compare Corollary~\ref{1-coro-prim-subst}. Choose any legal word of length $\ell$, $u = a a_2 \cdots a_{\ell} \in \mathcal{L}^{\ell}_{\vartheta}$, starting with the letter $a$. From the definition of the induced substitution, we find $\abs{\vartheta^n_{\ell}(u)(\omega)} = \abs{\vartheta^n(a)(\omega)}$, for all $\omega \in \Omega$. This also implies that $\mathbb{E} \abs{\vartheta^n_{\ell}(u)} = \mathbb{E} \abs{\vartheta^n(a)}$, for every $n \in \mathbb{N}$, yielding the second identity when combined with Eq.~\eqref{eq_expec_conv}.

Furthermore, we have $\abs{\vartheta^n_{\ell}(u)(\omega)}_v = \abs{\vartheta^n(a)(\omega)}_v + \mathcal{O}(\ell)$. This is because the last $\ell-1$ `letters' ($\in \mathcal{L}^{\ell}_{\vartheta}$) of $\vartheta^n_{\ell}(u)(\omega)$ are not entirely  included in $\vartheta^n(a)(\omega)$, whereas all others comprise the subwords of $\vartheta^n(a)(\omega)$. Thus, $\mathbb{E}\abs{\vartheta^n_{\ell}(u)}_v = \mathbb{E} \abs{\vartheta^n(a)}_v + \mathcal{O}(\ell)$. Since $\vartheta$ is expanding, it is $\lambda>1$ and so the error term $\mathcal{O}(\ell)$, when divided by $\mathbb{E} \abs{\vartheta^n(a)}$, gets arbitrarily small, as $n \rightarrow \infty$. Consequently, writing $\abs{\vartheta^n_{\ell}(u)}_v = \Phi_{\ell}(\vartheta^n_{\ell}(u))_v$,
\[
\lim_{n \rightarrow \infty} \frac{\mathbb{E}\abs{\vartheta^n(a)}_v}{\mathbb{E}\abs{\vartheta^n(a)}}
\, = \, \lim_{n \rightarrow \infty} \frac{\mathbb{E}\abs{\vartheta^n_{\ell}(u)}_v}{\mathbb{E} \abs{\vartheta^n_{\ell}(u)}} 
\, = \, R^{(\ell)}_v,
\]
by an application of the third identity in Lemma~\ref{1-lemma-exp-convergence}.
\end{proof}

For the rest of this paper, we set $\vartheta$ to be an expanding primitive random substitution on some finite alphabet $\mc A$ with $m$ letters. This allows us to use the powerful machinery derived from PF theory which we developed in this chapter. 

\section{word frequencies in inflation words}
\label{sec:word_freq-P-as}

The word frequency of some $v \in \mathcal{L}_{\vartheta}^{\ell}$ within a (generally larger) word $w \in \mathcal{L}_{\vartheta}$ will be denoted by
$
\nu_v(w) = \frac{|w|_v}{|w|}.
$
In this section, we will show that 
\[
\nu_v(\vartheta^n(a)) \, = \, \frac{| \vartheta^n(a)|_v}{| \vartheta^n(a) |}
\, \xrightarrow{n \rightarrow \infty} \, R_v^{(\ell)} \] 
holds almost surely, irrespective of the choice of $a \in \mc A$.  To this end, we will split the word $\vartheta^n(a) = \vartheta^{k}(\vartheta^{n-k}(a))$ into a large number of inflation words with level smaller than $n$. One technical obstacle is given by the fact that the (random) word $\vartheta^{n-k}(a)$ might comprise more than one type of letter, such that the action of $\vartheta^k$ on these letters cannot be interpreted as i.i.d. random words. This can be circumvented with the help of an elementary result. It basically establishes that if $\nu_v(\vartheta^n(u))$ deviates from $R_v^{(\ell)}$ by a certain amount, then there is a letter $a \in \mathcal{A}$, comprising a positive fraction of $u$, such that a similar statement holds if we regard only the inflation words $\vartheta^n(u_i)$ for those letters $u_i \triangleleft u$ that coincide with $a$. We set up a slightly more general form for this result to allow for its application also in a different context. 
\begin{lemma}
\label{lemma:shift-to-one-letter-type}
Let\/ $c = \sum_{i=1}^m c_i$ and\/ $d = \sum_{i=1}^m d_i$ with\/ $0 \leqslant c_i \leqslant d_i$ and\/ $1 \leqslant d_i$ for all\/ $i \in \{1,\ldots,m\}$. Further, suppose that 
\[
\left| \frac{c}{d} - K \right| \, > \, \delta ,
\]
for some\/ $0 \leqslant \delta,K \leqslant 1$. Then, there exists some\/ $j \in \{1,\ldots,m\}$ such that\/ $d_j \geqslant \frac{\delta}{2m} d$ and
\begin{equation}
\label{eq:single-type-deviation}
\biggl| \frac{c_j}{d_j} - K \biggr| \, \geqslant \, \frac{\delta}{2} .
\end{equation}
Further, for this choice of\/ $j$ and any\/ $f \in \mathbb{R}_+$,
\begin{equation}
\label{eq:modified-single-type-deviation}
\left| \frac{d_j}{f} - 1 \right| \, \geqslant \, \frac{\delta}{8} \quad \mbox{or} \quad \left| \frac{c_j}{f} - K \right| 
\, \geqslant \, \frac{\delta}{8}.
\end{equation}
\end{lemma}
\begin{proof}
Suppose Eq.~\eqref{eq:single-type-deviation} does not hold and set  
\[ S_1 \, = \, \left\{ 1\leqslant i \leqslant m \Mid | c_i - K d_i | \geqslant \frac{\delta}{2} d_i \right\}
\]
and $S_2 = \{1,\ldots,m\} \setminus S_1$. Then, $i \in S_1$ implies $d_i < \frac{\delta}{2m} d$. Note that generally $| c_i - K d_i| \leqslant d_i$. Consequently,
\begin{align*}
| c - Kd| 
& \, \leqslant \, \sum_{i\in S_1} \left| c_i - K d_i \right| + \sum_{i \in S_2} \left| c_i - K d_i \right|
\, < \, \sum_{i \in S_1} d_i + \sum_{i \in S_2} \frac{\delta}{2}d_i
 \, < \, \sum_{i \in S_1} \frac{\delta}{2m} d + \frac{\delta}{2} d  \, \leqslant \, \delta d,
\end{align*}
in contradiction to the assumption. This proves Eq.~\eqref{eq:single-type-deviation}. Assuming $|d_j/f -1| \leqslant \delta/8$, we obtain, 
\[
\frac{\delta}{4} \, \leqslant \, \frac{d_j}{f} \frac{\delta}{2} 
\, \leqslant \, \left| \frac{c_j}{f} - K \frac{d_j}{f} \right| 
\, \leqslant \, \left| \frac{c_j}{f} - K \right| + K\left| 1 - \frac{d_j}{f} \right| 
\, \leqslant \, \left| \frac{c_j}{f} - K \right| + \frac{\delta}{8} ,
\]
and thus Eq.~\eqref{eq:modified-single-type-deviation} holds.
\end{proof}

\begin{lemma}
\label{lemma:word-freq-prep}
Le\/ $v \in \mc L_{\vartheta}^{\ell}$ and\/ $\varepsilon >0$. There is a\/ $k_0 \in \mathbb{N}$ such that, for all\/ $k \geqslant k_0$, there is a\/ $C>0$ and\/ $n_0 \in \N$ with the following property. If\/ $n \geqslant n_0$, we have
\[
\mathbb{P} \left[ \bigl| \nu_v(\vartheta^k(u)) - R_v^{(\ell)} \bigr| > \varepsilon \right] \, \leqslant \, \me^{-C n} ,
\] 
for all $u \in \mc L_{\vartheta}^n$.
\end{lemma}

\begin{proof}
Let us start with some arbitrary $n \in \mathbb{N}$ and $u \in \mc L_{\vartheta}^n$. For any letter $a \in \mc A$, denote the positions of the letter $a$ within the word $u$ by $P_a = \{ 1 \leqslant i \leqslant n \Mid u_i = a \}$, possibly empty, and $n_a = \card P_a= |u|_a$. Assume $\card \mc A = m$ and 
\[
\left| \frac{|\vartheta^k(u)|_v}{|\vartheta^k(u)|} - R_v^{(\ell)} \right| \, > \, \varepsilon.
\] 
Note that, in contrast to the total length $|\vartheta^k(u)| = \sum_{j=1}^n |\vartheta^k(u_j)|$, for the occurrences of $v$ within $\vartheta^k(u)$ it is not enough to count how often $v$ appears in any inflation word,
\[
| \vartheta^k(u) |_v  \, = \, \sum_{j=1}^n | \vartheta^k(u_j) |_v + \operatorname{Ovl},
\]
where the additional term $\operatorname{Ovl}$ (in general a random variable) gives the number of times that the position of $v$ within $\vartheta^k(u)$ overlaps two (or more) neighbouring inflation words. A situation where such an overlap occurs is depicted in Figure~\ref{Fig:overlap-position}. Since there are only $n-1$ different boundaries of inflation words of the form $\vartheta^k(u_i)$, the bound $0 \leqslant \operatorname{Ovl} < n \ell$ follows immediately. Setting $d = |\vartheta^k(u)|$ and $c =\sum_{j=1}^n | \vartheta^k(u_j) |_v$, one of the following two cases needs to hold,
\begin{figure}
\begin{tikzpicture}
\node at (5,0) {$ | \underbrace{\cdots \cdots}_{\vartheta^k(u_1)} | \underbrace{\cdots \cdot}_{\vartheta^k(u_2)} | \cdots \cdots \cdot | \cdot \cdots \cdot \cdot | \cdot \cdots | \underbrace{ \cdots \cdots \cdots}_{\vartheta^k(u_n)} | \quad = \vartheta^k(u)$};
\draw (3.65,0.7) -- node[above]{$\scriptstyle v$} (4.3,0.7);
\draw (3.65,0.6) -- (3.65,0.8)  (4.3,0.6) -- (4.3,0.8);
\end{tikzpicture}
\caption{Example of a case where $v$ is not entirely included in any of the inflation words $\vartheta^k(u_j)$ but still occurs as a subword of $\vartheta^k(u)$.}
\label{Fig:overlap-position}
\end{figure}
\[
 \mbox{(I)} \quad \left| \frac{c}{d} - R_v^{(\ell)} \right| \, > \, \frac{\varepsilon}{2} \quad \quad \mbox{or} \quad \quad 
\mbox{(II)} \quad \ell n \, > \, \frac{\varepsilon}{2} \bigl|\vartheta^k(u) \bigr| .
\]
Let us first treat case (I). We want to translate this to a similar statement that involves only one type of letter as a starting point for the random inflation $\vartheta^k$; compare Figure~\ref{Fig:2-step-inflation} for an illustrative example.
\begin{figure}
\begin{tikzpicture}
\node at (10,7.5) {$\cdot \cdot a \cdots a \cdots \cdot a \cdots$};
\node at (13,7.5) {$=u$};
\draw (8.1,7.2) -- (6.95,5.5)  (11.8,7.2) -- (13.0,5.5);

\draw[dashed] (8.8,7.2) -- (8,5.5) (9.77,7.2) -- (9.7,5.5) (10.9,7.2) -- (11.5,5.5);

\node at (10,5) {$\cdots \vartheta^k(a) \cdots \cdot \vartheta^k(a) \cdots  \cdot \cdot \vartheta^k(a) \cdots \cdot \cdot$};
\node at (14.2,5) {$=\vartheta^k(u)$};
\end{tikzpicture}
\caption{Inflation procedure of $\vartheta^k$ on $u$. The occurrences of some $a \in \mc A$ in $u$ are highlighted. These are mapped to i.i.d random words, distributed like $\vartheta^k(a)$, which appear as subwords of $\vartheta^k(u)$.}
\label{Fig:2-step-inflation}
\end{figure}
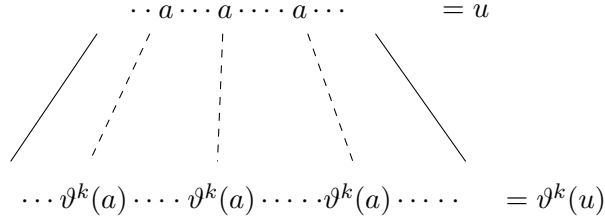
With the identifications $\delta = \frac{\varepsilon}{2}$, $K = R_v^{(\ell)}$, $c_a = \sum_{j \in P_a} |\vartheta^k(u_j)|_v$, $d_a = \sum_{j \in P_a} |\vartheta^k(u_j)|$ we can apply Lemma~\ref{lemma:shift-to-one-letter-type}. That is, for all realisations of the random variables that we consider, there is some letter $a \in \mc A$ such that $\sum_{j \in P_a} |\vartheta^k(u_j)| \geqslant \frac{\varepsilon}{4m} |\vartheta^k(u)|$ and (choosing $f = \lambda^k L_a n_a$, in the notation of Lemma~\ref{lemma:shift-to-one-letter-type}),
\[
\biggl| \frac{1}{n_a} \sum_{j \in P_a} \frac{ |\vartheta^k(u_j)|_v}{\lambda^k L_a} - R_v^{(\ell)} \biggr| \, \geqslant \, \frac{\varepsilon}{16}
 \quad \mbox{or} \quad 
 \biggl| \frac{1}{n_a} \sum_{j \in P_a} \frac{ |\vartheta^k(u_j)|}{\lambda^k L_a} - 1 \biggr| \, \geqslant \, \frac{\varepsilon}{16},
\]
which we call case (Ia) and (Ib), respectively. Setting $\mu_a = \max \{ k \in \N \Mid \mathbb P[|\vartheta(a)| =k]>0 \}$ and $\mu = \max_{a \in \mc A} \mu_a$ as the maximal possible inflation factor, we have the crude estimate,
\[
\mu^k n_a \, \geqslant \, \sum_{j \in P_a} |\vartheta^k(u_j)| \, \geqslant \, \frac{\varepsilon}{4m} |\vartheta^k(u)| \, \geqslant \, \frac{\varepsilon}{4m} n,
\]
that is, there is some constant $C_1(k,\varepsilon) = C_1 > 0$ such that $n_a \geqslant C_1 n$.
Let us define the random variables
\[
X^k_{(j)} \, = \, \frac{ |\vartheta^k(u_j)|_v}{\lambda^k L_a} \quad \mbox{and} \quad Y^k_{(j)} \, = \, \frac{ |\vartheta^k(u_j)|}{\lambda^k L_a},
\]
and note that each of the corresponding families with $j \in P_a$ comprise i.i.d. random variables with the same distribution $X^k$ and $Y^k$, respectively.
By Lemma~\ref{1-lemma-exp-convergence} and Proposition~\ref{thm-expec-conv},
\begin{align*}
\mathbb{E} Y^k_{(j)} \, \xrightarrow{k \rightarrow \infty} \, 1 \quad \mbox{and} \quad \mathbb{E}X^k_{(j)}  \, \xrightarrow{k \rightarrow \infty} \, R_v^{(\ell)},
\end{align*}
for any $j \in P_a$. 
We can thus find a $k^{(1)}_0 \in \mathbb{N}$ such that, for all $k \geqslant k^{(1)}_0$,
\[
\biggl| \frac{1}{n_a} \sum_{j \in P_a} X^k_{(j)} - \mathbb{E} X^k \biggr| \, \geqslant \, \frac{\varepsilon}{32}
\quad \mbox{or} \quad \biggl| \frac{1}{n_a} \sum_{j \in P_a} Y^k_{(j)} - \mathbb{E}  Y^k \biggr| \, \geqslant \, \frac{\varepsilon}{32}.
\]
Let us fix such a $k$. Then, the joint distribution of the family $\bigl(X^k_{(j)} \bigr)_{j \in P_a}$ does not depend on the structure of $u$ or the exact form of $P_a$, but only on $n_a = \card P_a$. Also, since $X^k$ can take only finitely many values, the bound 
$
\mathbb{E} \me^{tX^k} < \infty
$
is trivial for all $t \in \mathbb{R}$. By Cram\'{e}r's theorem on  large deviations \cite[Thm.~I.4]{hollander}, we therefore obtain,
\[
\lim_{n_a \rightarrow \infty} \frac{1}{n_a} \log \mathbb{P} \biggl[\biggl| \frac{1}{n_a} \sum_{j \in P_a} X^k_{(j)} - \mathbb{E} X^k \bigg| \geqslant \frac{\varepsilon}{32} \biggr] \, < \, 0,
\]
and analogously for the family $\bigl(Y^k_{(j)} \bigr)_{j \in P_a}$. Thus, for large enough $n$ and any $u \in \mc L_n$,
\[
\mathbb{P}[\mbox{(I)}] \, \leqslant \, \sum_{a \in \mc A} \mathbb{P}[\mbox{(Ia) holds for } a] + \sum_{a \in \mc A} \mathbb{P}[\mbox{(Ib) holds for } a]
\, \leqslant \, \sum_{a \in \mc A} \me^{-\widetilde{C} n_a } 
\, \leqslant \, | \mc A| \me^{ - \widetilde{C} C_1 n}.
\]
We now turn to case (II). This can be rephrased as 
\[
\frac{\sum_{a \in \mc A} |u|_a}{\sum_{a \in \mc A} \sum_{j \in P_a} | \vartheta^k(u_j)|} \, > \, \frac{\varepsilon}{2 \ell},
\]
such that we can again apply Lemma~\ref{lemma:shift-to-one-letter-type}, this time with $\delta= \frac{\varepsilon}{2 \ell}$, $K = 0$, $c_a = |u_a|$ and $d_a = \sum_{j \in P_a}|\vartheta^k(u_j)|$, yielding the existence of some $a \in \mc A$ with $\sum_{j \in P_a}|\vartheta^k(u_j)| \geqslant \frac{\epsilon}{4 \ell m} |\vartheta^k(u)|$ and, recalling $n_a = |u|_a$,
\begin{equation}
\label{Eq:5-upper-inflation-bound}
\frac{1}{n_a} \sum_{j \in P_a} | \vartheta^k(u_j)| \, < \, \frac{4\ell}{\varepsilon}.
\end{equation}
Note that all the random words $\vartheta^k(u_j)$ are distributed as $\vartheta^k(a)$ and again i.i.d. for $j \in P_a$.
At the same time,
\[
n_a \, > \, \frac{\varepsilon}{4 \ell} \sum_{j \in P_a} | \vartheta^k(u_j) | \, \geqslant \, \frac{\varepsilon^2}{16 \ell^2 m} | \vartheta^k(u) | \, \geqslant \, \frac{\varepsilon^2}{16 \ell^2 m} n.
\]
Since $\mathbb{E}|\vartheta^k(u_j)| \xrightarrow{k \to \infty} \infty$, as $k \rightarrow \infty$, we can find a $k_0^{(2)}$ such that, for all $k \geqslant k_0^{(2)}$, Eq.~\eqref{Eq:5-upper-inflation-bound} implies
\[
\frac{1}{n_a} \sum_{j \in P_a} | \vartheta^k(u_j)| \, < \, \frac{1}{2} \mathbb{E}|\vartheta^k(a)|.
\]
Similarly to case (I), we can apply basic large deviation results to conclude that the probability for this to happen is bounded by $\me^{-C_2 n}$, for some $C_2 > 0$ and $n \in \mathbb{N}$ large enough.
\\Finally, taking $k_0 = \max\{k_0^{(1)}, k_0^{(2)}\}$ and adjusting the constant in the exponent, $k \geqslant k_0$ implies
\[
\mathbb{P} \left[ \bigl| \nu_v(\vartheta^k(u)) - R_v^{(\ell)} \bigr| > \varepsilon \right] \, \leqslant \, \mathbb{P}[\mbox{(I)}] + \mathbb{P}[\mbox{(II)}] \, \leqslant \, \me^{-C n}.
\]
for large enough $n \in \mathbb{N}$.
\end{proof}

\begin{prop}
\label{prop:as-word-frequencies}
Let\/ $v \in \mc L_{\vartheta}^{\ell}$ be a $\vartheta$-legal word of length\/ $\ell \in\N$. Then, in the limit of large inflation words, the word frequency of\/ $v$ exists and is the same for almost all realisations. It coincides with the corresponding entry\/ $R_v^{(\ell)}$ of the right PF eigenvector of the induced substitution matrix. More precisely, for any $a \in \mathcal{A}$,
\[
\lim_{n \to \infty} \nu_v (\vartheta^n(a)) = R_v^{(\ell)}
\]
holds almost surely.
\end{prop}
\begin{proof}
Let $\varepsilon > 0$ and $a \in \mc A$. We are going to establish that
\begin{equation}
\label{eq:freq-var-summable}
\sum_{n=1}^{\infty} \mathbb{P} \left[ \bigl| \nu_v(\vartheta^n(a)) - R_v^{(\ell)} \bigr| > \varepsilon \right] \, < \, \infty.
\end{equation}
From this, the assertion follows by a standard application of the Borel--Cantelli lemma.

Due to Proposition~\ref{length-as-infty}, there are some $\widetilde{C},K>0$ such that $\mathbb{P} [ | \vartheta^n(a) | < K n] \leqslant \me^{-\widetilde{C} n}$ for large enough $n \in \mathbb{N}$. Choose $C>0$ and a fixed $k \geqslant k_0$ as in Lemma~\ref{lemma:word-freq-prep}. Then, for large enough $n \in \mathbb{N}$, 
\begin{align*}
\mathbb{P} \left[ \bigl| \nu_v(\vartheta^{k+n}(a)) - R_v^{(\ell)} \bigl| > \varepsilon \right] 
& \, \leqslant \, \me^{- \tilde{C}n} + \sum_{u \in \mc L ,\, |u| \geqslant Kn} \mathbb{P} \left[ \bigl| \nu_v(\vartheta^{k}(u)) - R_v^{(\ell)} \bigl| > \varepsilon \right] \mathbb{P}[\vartheta^{n}(a) = u] 
\\ & \, \leqslant \, \me^{-\widetilde{C} n}+ \me^{- CKn},
\end{align*}
by Lemma~\ref{lemma:word-freq-prep}. This verifies Eq.~\eqref{eq:freq-var-summable}.
\end{proof}

\begin{remark}
The statement of Proposition~\ref{prop:as-word-frequencies} for the special case\/ $\ell =1$ (frequency of letters) follows from a standard result on branching processes. More precisely, performing the Abelianisation, we obtain a vector-valued Markov process\/ $(\Phi(\vartheta^n(a)))_{n \in \mathbb{N}}$ which forms a multitype Galton--Watson (GW) process as described in \cite{mode} and \cite{athreya}. From \cite[Ch.~V.6]{athreya}, we find that, for every\/ $a \in \mc A$,
\begin{equation}
\label{eq:GW-convergence}
\lim_{n \rightarrow \infty} \frac{\Phi(\vartheta^n(a))}{\lambda^n} 
\, = \, R W_a 
\end{equation}
holds almost surely, where\/ $W_a$ is a non-negative random variable with\/ $\mathbb{E}[W_a] = L_a$ and\/ $\mathbb{P}[W_a = 0] = 0$. Note that\/ $( \Phi_{\ell}(\vartheta^n_{\ell}(a)))_{n \in \mathbb{N}}$ for\/ $\ell \geqslant 2$ does \emph{not} form a multitype GW process, because neighbouring words have some overlap which is necessarily mapped to the same word under\/ $\vartheta$. This violates the requirement for a GW process that each individual produces offspring independently. However, combining Proposition~\ref{prop:as-word-frequencies} with Eq.~\eqref{eq:GW-convergence} we obtain that for any\/ $a \in \mc A$ and\/ $v \in \mc L_{\vartheta}^{\ell}$,
\[
\lim_{n \to \infty} \frac{|\vartheta^n(a)|_v}{\lambda^n}
\, = \, \lim_{n \to \infty} \frac{|\vartheta^n(a)|_v}{|\vartheta^n(a)|} \frac{|\vartheta^n(a)|}{\lambda^n} 
\, = \, R_v^{(\ell)} W_a,
\]
holds almost surely. Note that, despite our preceding words of precaution, this result (and thereby Proposition~\ref{prop:as-word-frequencies}) can also be obtained by generalising ideas from the theory of GW processes  \cite[Thm.~2]{peyriere}. 
\end{remark}

\section{Ergodic frequency measures on the hull}
\label{sec:main-result}

From the word frequencies in inflation words, as established in the last section, we want to define a shift-invariant measure on the hull $\mathbb{X}_{\vartheta}$ of the substitution in a consistent way. We obtain this by defining the measure first on a convenient class of subsets.
A cylinder set on $\mathbb{X}_{\vartheta}$ is of the form $\mathcal{Z}_{[k,m]}(v) = \{ x \in \mathbb{X}_{\vartheta} \Mid x_{[k,m]} = v \}$, for some $k\leqslant m \in \mathbb{Z}$ and $v \in \mc L_{\vartheta}^{m-k+1}$. We call $\mc B_{\vartheta}$ the Borel sigma-algebra with respect to the topology generated by the cylinder sets.

\begin{remark}
\label{rem:semi-algebra}
It is worth noticing that\/ $\mc B_{\vartheta}$ is already generated by the smaller class of those cylinder sets that specify the position\/ $0$:
\[
\mathfrak{Z}_{0}(\mathbb{X}_{\vartheta}) = \{ \mathcal{Z}_{[k,m]}(v) \Mid k\leqslant 0 \leqslant m, v \in \mc L_{\vartheta}^{m-k+1} \} \cup \{ \mathbb{X}_{\vartheta}, \varnothing \}.
\]
The class\/ $\mathfrak{Z}_{0}(\mathbb{X}_{\vartheta})$ has the convenient property of forming a semi-algebra on\/ $\mathbb{X}$ (it contains the full space, is stable under intersections, and complements can be written as finite disjoint unions).
\end{remark}

\begin{definition}
A map\/ $\mu: \mathfrak{S} \rightarrow [0,\infty]$ on a semi-algebra\/ $\mathfrak{S}$ is called a \emph{measure} on\/ $\mathfrak{S}$, if the following conditions hold:
\begin{enumerate}
\setlength\itemsep{0.5em}
\item If\/ $\dot{\bigcup}_{i=1}^{\infty} A_i \in \mathfrak{S}$ is a  disjoint union of elements\/ $A_i \in \mathfrak{S}$, then\/ $\mu(\dot{\bigcup}_{i=1}^{\infty} A_i) = \sum_{i=1}^{\infty} \mu(A_i)$,
\item $\mu(\varnothing) = 0$.
\end{enumerate}
\end{definition}

Often, this is also called a premeasure in the literature. We know from \cite[Cor. 2.4.9]{partha} and \cite[Prop. 2.5.1]{partha} that this can be extended uniquely to a measure on the $\sigma$-algebra generated by $\mathfrak{S}$ if $\mu$ is $\sigma$-finite on  $\mathfrak{S}$. Comparing this with Remark~\ref{rem:semi-algebra}, we find that it is enough to specify a measure on $\mathfrak{Z}_0(\mathbb{X}_{\vartheta})$ in order to uniquely define a measure on $(\mathbb{X}_{\vartheta}, \mathcal{B}_{\vartheta})$. 

\begin{prop}
\label{prop:mu-on-semi-algebra}
Let\/ $\mu_{\vartheta}:\mathfrak{Z}_0(\mathbb{X}_{\vartheta}) \rightarrow [0,1] $ be defined via\/ $\mu_{\vartheta}(\varnothing)=0$, $\mu_{\vartheta}(\mathbb{X}_{\vartheta}) =1 $ and\/ $\mu_{\vartheta} (\mathcal{Z}_{[k,m]}(v)) = R^{(\ell)}_v$, for any\/ $k\leqslant 0 \leqslant m \in \mathbb{Z}$ and\/ $v \in \mathcal{L}^{\ell}_{\vartheta}$ with\/ $|v| = m-k+1$.  Here,\/ $R^{(\ell)}$ denotes the right PF-eigenvector of the induced substitution\/ $\vartheta_{\ell}$. Then,\/ $\mu_{\vartheta}$ defines a $\sigma$-finite measure on\/ $\mathfrak{Z}_0(\mathbb{X}_{\vartheta})$ and thereby also on\/ $(\mathbb{X}_{\vartheta}, \mathcal{B}_{\vartheta})$.
\end{prop}

As a first step, we establish a consistency relation for the word frequencies which ensures the finite additivity of $\mu_{\vartheta}$.

\begin{lemma}\label{lemma-freq-consistency} 
Suppose\/ $v = v_1 \cdots v_{\ell_0} \in \mathcal{L}_{\vartheta}$, and fix an arbitrary\/ $\ell \in \mathbb{N}$, with\/ $\ell \geqslant \ell_0$ and a position\/ $k \in \{1, \ldots, \ell - \ell_0+1 \}$. If\/ $\mathcal{G}^{\ell}_{k}(v) = \{ u \in \mathcal{L}^{\ell}_{\vartheta} \Mid u_{[k,k+\ell_0-1}] = v \}$, then\/ \[R^{(\ell_0)}_v = \sum_{u \in \mathcal{G}^{\ell}_k(v)} R^{(\ell)}_u.\]
\end{lemma}
\begin{proof}
Fix $\omega \in \Omega$ and $a \in \mathcal{A}$ such that $\nu_u(\vartheta^n(a)(\omega))$ converges to $R^{(\abs{u})}_u$ as $n \to \infty$, for all $u \in \mathcal{L}_{\vartheta}$. Now, 
\[
\abs{\vartheta^n(a)(\omega)}_v \, = \, \operatorname{card} \, \bigl\{ j \in \{1, \ldots, \abs{\vartheta^n(a)(\omega)} - \ell_0 +1 \} \Mid \vartheta^n(a)(\omega)_{[j,j+\ell_0 -1]} = v \bigr\}.
\] 
Note that, for $j \geqslant k $, $\vartheta^n(a)(\omega)_{[j,j+\ell_0 -1]} = v$ if and only if $\vartheta^n(a)(\omega)_{[j-k+1,j -k +\ell ]} = u$ for some $u \in \mathcal{G}_k^{\ell}(v)$, as long as $j -k + \ell \leqslant \abs{\vartheta^n(a)(\omega)}$. That is, if we count the cardinality of the corresponding set, we miss at most $\ell - \ell_0$ subwords of type $v$ in $\vartheta^n(a)(\omega)$. More precisely, with $i = j-k$,
\begin{align*}
\abs{\vartheta^n(a)(\omega)}_v &\, = \, \operatorname{card} \mathop{\dot{\bigcup}}_{u \in \mathcal{G}^{\ell}_k(v)} \{ i \in \{0, \ldots,\abs{\vartheta^n(a)(\omega)} - \ell \} \Mid  \vartheta^n(a)(\omega)_{[i+1,i+\ell ]} = u  \} + \mathcal{O}(\ell - \ell_0)
\\ & \, = \, \sum_{u \in \mathcal{G}^{\ell}_k(v)} \abs{\vartheta^n(a)(\omega)}_u + \mathcal{O}(\ell - \ell_0). 
\end{align*}
Dividing both sides by $\abs{\vartheta^n(a)(\omega)}$ and taking the limit $n \rightarrow \infty$ yield the desired relation.
\end{proof}

\begin{proof}[Proof of Proposition~\textnormal{\ref{prop:mu-on-semi-algebra}}]
First, note that every $\mathcal{Z} \in \mathfrak{Z}_0(\mathbb{X}_{\vartheta})$ is compact as a closed subset of a compact space. Hence, every countable disjoint union $\dot{\bigcup}_{i=1}^{\infty} \mc Z_i = \mathcal{Z} \in \mathfrak{Z}_0(\mathbb{X}_{\vartheta})$ is in fact finite, that is  $\dot{\bigcup}_{i=1}^{\infty} \mc Z_i = \dot{\bigcup}_{j\in I} \mc Z_{j}$ for some finite index set $I$. It thus suffices to show that $\mu_{\vartheta}$ is finitely additive on $\mathfrak{Z}_0(\mathbb{X}_{\vartheta})$.

We can rewrite any cylinder set $ \mc Z_{[j,m]}(v)$ as a union of cylinder sets that specify a larger set of positions
\begin{equation}\label{eq-cylinder-expansion}
\mc Z_{[j,m]}(v) \, = \, \mathop{\dot{\bigcup}}_{u \in \mathcal{G}_k^{2n+1}(v)} \mathcal{Z}_{[-n,n]}(u),
\end{equation}
for any $n \geqslant \max \{ \abs{j},\abs{m} \}$ and $k = j+n+1 $. Lemma~\ref{lemma-freq-consistency} yields
\begin{equation}\label{eq-add-on-cylinder-exp}
\mu_{\vartheta} (\mathcal{Z}_{[j,m]}(v)) \, = \, R^{(m-j+1)}_v 
\, = \, \sum_{u \in \mathcal{G}^{2n+1}_k(v)} R^{(2n+1)}_u 
\, = \, \sum_{u \in \mathcal{G}_k^{2n+1}} \mu_{\vartheta}(\mathcal{Z}_{[-n,n]}(u)),
\end{equation}
providing the additivity of $\mu_{\vartheta}$ on disjoint unions of the form in Eq.~\eqref{eq-cylinder-expansion}. The same kind of relation holds for $\mc Z = \mathbb{X}_{\vartheta}$ if we replace $\mc G_k^{2n+1}(v)$ by $\mc L_{\vartheta}^{2n+1}$.

The finite additivity of $\mu_{\vartheta}$ on disjoint unions of cylinder sets of the form  $\dot{\bigcup}_{j\in I} \mc Z_{j} = \mc Z$ follows in a straightforward manner. Just consider a common refinement into cylinder sets of the form $\mc Z_{[-n,n]}(u)$ for large enough $n \in \mathbb{N}$. Hence, $\mu_{\vartheta}$ is a ($\sigma$-)finite measure on $\mathfrak{Z}_0(\mathbb{X}_{\vartheta})$.
\end{proof}

\begin{remark}
It is clear by definition that\/ $\mu_{\vartheta}$ is shift-invariant on the cylinder sets of the semi-algebra\/ $\mathfrak{Z}_0(\mathbb{X}_{\vartheta})$. This also implies that\/ $\mu_{\vartheta}$ is shift-invariant as a measure on\/ $(\mathbb{X}_{\vartheta}, \mathcal{B}_{\vartheta})$ --- compare \cite[Thm.~1.1]{walters}. We are therefore left with a shift invariant measure space\/ $(\mathbb{X}_{\vartheta}, \mathcal{B}_{\vartheta}, \mu_{\vartheta})$. 
\end{remark}

As an important step towards ergodicity, we give a sufficient condition for the almost sure convergence of word frequencies with respect to the measure $\mu_{\vartheta}$. For the following, we fix an arbitrary averaging sequence of intervals $I_n = [a_n,b_n]$ with $a_n \leqslant b_n$ and $a_n, b_n \in \mathbb{Z}$ for all $n\in \mathbb{N}$ such that the length of the intervals $d_n = b_n - a_n +1$ is strictly monotonously increasing.

\begin{lemma}
Let\/ $(\mathbb{X}_{\vartheta}, \mathcal{B}_{\vartheta}, \mu_{\vartheta})$ be the measure space introduced above. For a given word\/ $v \in \mathcal{L}^{\ell}_{\vartheta}$ and\/ $\varepsilon > 0$, let\/ $C^{n}_{\varepsilon} := \{ u \in \mathcal{L}^{n}_{\vartheta} \Mid \abs{\nu_v(u) - R^{(\ell)}_v } > \varepsilon \}$. Suppose that, for every\/ $\varepsilon > 0$,
\begin{equation}\label{eq-summability}
\sum_{n=1}^{\infty} \sum_{u \in C_{\varepsilon}^{n}} R^{(n)}_u \, <  \, \infty.
\end{equation}
Then, for almost every element\/ $y \in \mathbb{X}_{\vartheta}$, the frequency of\/ $v$ in\/ $y$ with respect to the averaging sequence $\{[a_n,b_n]\}_{n \in \mathbb{N}}$ is well-defined and given by\/ $R^{(\ell)}_v$. That is, there exists a set of full measure\/ $A^v \subseteq \mathbb{X}_{\vartheta}$ such that, for any\/ $y \in A^v$, we have
\begin{equation}\label{eq-freq-in-hull}
\lim_{n \rightarrow \infty} \nu_v \bigl(y^{}_{[a_n,b_n]}\bigr)
\, = \, R^{(\ell)}_v.
\end{equation}
\end{lemma}
\begin{proof}
Consider the random variables
\[
X_n:\mathbb{X}_{\vartheta}\to \R,\quad\ y\mapsto \nu_v \bigl(y^{}_{[a_n,b_n]}\bigr) = \frac{|y_{[a_n,b_n]}|_v}{d_n}, 
\]
for all $n\in\N$, and the constant random variable $X:\mathbb{X}_{\vartheta}\to \R$, $y\mapsto R_v^{(\ell)}$. It is a well-known fact (by an application of the Borel--Cantelli lemma) that $(X_n)_{n\in\N}$ converges $\mu_{\vartheta}$-almost surely to $X$ if
\[
\sum_{n=1}^{\infty} \mu_{\vartheta}\big(\{y\in\mathbb{X}_{\vartheta}\, |\, |X_n(y)-X(y)|>\varepsilon\}\big) \, < \, \infty,
\]
for every $\varepsilon>0$. Therefore, our claim follows from
\[
\begin{split}
\mu_{\vartheta}\big(\{y\in\mathbb{X}_{\vartheta}\, |\, |X_n(y)-X(y)|>\varepsilon\}\big) 
   & \, = \, \mu_{\vartheta}\big(\{y\in\mathbb{X}_{\vartheta}\, |\, y_{[a_n,b_n]}\in 
      C_{\varepsilon}^{d_n}\}\big)       \\
   & \, = \, \mu_{\vartheta}\big(\{y\in\mathbb{X}_{\vartheta}\, |\, y\in \dot{\bigcup}_{u\in
      C_{\varepsilon}^{d_n}} \mathcal{Z}_{[a_n,b_n]}(u) \}\big) \\
   & \, = \, \mu_{\vartheta}\biggl(\dot{\bigcup}_{u\in C_{\varepsilon}^{d_n}} 
        \mathcal{Z}_{[a_n,b_n]}(u) \biggr) = \sum_{u\in C_{\varepsilon}^{d_n}} \mu_{\vartheta} 
         \big(\mathcal{Z}_{[a_n,b_n]}(u)\big)                  \\
   & \, = \, \sum_{u\in C_{\varepsilon}^{d_n}} R_u^{(d_n)}
    \, \leqslant \, \sum_{u \in C_{\varepsilon}^{n}} R_u^{(n)},
\end{split}
\]
where the last step is a consequence of the assumption that $(d_n)_{n \in \mathbb{N}}$ is a strictly increasing sequence in $\mathbb{N}$.
\end{proof}

\begin{remark}
The proof of the ergodicity result for a family of random noble means subsitutions in \cite{moll} follows a somewhat different approach.   There, for a given $\ell \in \mathbb{N}^{}_0$ and $k,m \in \mathbb{Z}$, the subwords $X_{[k,k+\ell]}$ and $X_{[m,m+\ell]}$ of a $\mu_{\vartheta}$-distributed random word $X$ are interpreted as i.i.d.\ random words, provided that $k$ and $m$ are separated by a certain distance. The small gap in the proof, mentioned in the introduction, amounts to the fact that independence for the given, fixed distance does not hold and should be replaced by asymptotic independence as the distance approaches infinity. However, the varied correlation structure of different examples of random substitutions precludes using a similar approach in the general case. We therefore speak of i.i.d.\ random words only in the context of the Markov measure $\mathbb{P}$ and use the estimate in Lemma~\ref{lemma:word-freq-prep} to relate this to properties of $\mu_{\vartheta}$.
\end{remark}

Next, we want to show that Eq.~\eqref{eq-summability} indeed holds for any expansive primitive random substitution. The intuition behind this is the following. For a word $u$ to be in $C_{\varepsilon}^n$, it must be exceptional regarding the occurrences of the subword $v$. The sum of the frequencies of such words $u$ can not be too large since otherwise it would contradict the well-defined frequency of $v$ in the typical limit words of the random substitution Markov process. The idea is to exhaust exceptional words $u$ with inflation words and to show that $u \in C_{\varepsilon}^n$ essentially requires a positive fraction of these inflation words to have exceptional relative words frequencies as well. We can then employ the large deviation estimates in Lemma~\ref{lemma:word-freq-prep} to show that these events must have summable probabilities.

\begin{prop} \label{prop:mu-as-existence-of-word-frequencies}
Let\/ $v \in \mathcal{L}^{\ell}_{\vartheta}$ be a $\vartheta$-legal word of length\/ $\ell$ and\/ $\{I_n\}_{n \in \mathbb{N}}$ an averaging sequence of strictly increasing length. Then, for $\mu_{\vartheta}$-almost every element\/ $y \in \mathbb{X}_{\vartheta}$, the frequency of\/ $v$ in\/ $y$ with respect to\/ $\{I_n\}_{n \in \mathbb{N}}$ is well-defined and given by\/ $R^{(\ell)}_v$.
\end{prop}
\begin{proof}
It suffices to show that Eq.~\eqref{eq-summability} indeed holds for all $\varepsilon >0$. 
Let $v \in \mathcal{L}^{\ell}_{\vartheta}$ be fixed in the following. Because of Proposition~\ref{1_prop_diagram}, it is for $x,u \in \mathcal{L}^{n}_{\vartheta}$ and $k\in\N$: $(M_{\vartheta_{n}}^k)_{xu} = \mathbb{E} \Phi_n (\vartheta^k_{n}(u))_x$. Hence, we can rewrite the frequency of $x \in C_{\varepsilon}^{n}$ as 
\[
 R_x^{(n)} \, = \, \frac{1}{\lambda^k} \left(M^k_{\vartheta_{n}} R^{(n)} \right)_x \, = \, \frac{1}{\lambda^k} \sum_{u \in \mathcal{L}^{n}_{\vartheta}} \mathbb{E} \Phi_n (\vartheta^k_{n}(u))_x \, R_u^{(n)} 
  \, = \, \sum_{u \in \mathcal{L}^n_{\vartheta}}  R_u^{(n)} \underbrace{ \frac{1}{\lambda^k}  \sum_{V \in \mathcal{D}_n} \mathbb{P}[\vartheta_n^k(u) = V] \abs{V}_x}_{=: Q_n^k(u,x)} ,
\]
for any $k \in \mathbb{N}$. Thus, with $Q_n^k(u,C^n_{\varepsilon}) := \sum_{x \in C^n_{\varepsilon}} Q_n^k(u,x)$, it is
\begin{equation}\label{Eq:sum_deviations}
\sum_{x \in C^n_{\varepsilon}} R_x^{(n)} \, = \, \sum_{u \in \mathcal{L}^n_{\vartheta}}  R_u^{(n)} Q_n^k(u,C^n_{\varepsilon}). 
\end{equation}
Given a large enough $k$, our strategy is to find a bound for $Q_n^k(u,C^n_{\varepsilon})$, uniform in $u$, that is small enough to be summable over $n$. Note that 
\begin{align*}
Q_n^k(u,C^n_{\varepsilon}) &\, = \, \frac{1}{\lambda^k}  \sum_{V \in \mathcal{D}_n} \sum_{j=1}^{\abs{V} } \mathbb{P}[\vartheta_n^k(u) = V] \delta_{V^{}_j}(C_{\varepsilon}^n) 
\\ &\, = \, \frac{1}{\lambda^k} \sum_{m=1}^{\infty} \sum_{j=1}^{m} \mathbb{P}[\vartheta^k_n(u)_j \in C_{\varepsilon}^n \wedge \abs{\vartheta^k_n(u)} = m].
\end{align*}
This is still a finite sum, since $\abs{\vartheta_n^k(u)}$ can take only finitely many values. 
We can exchange the order of summation and express this as
\[
Q_n^k(u,C^n_{\varepsilon})\, = \, \frac{1}{\lambda^k} \sum_{j=1}^{\infty} \mathbb{P}[\vartheta^k_n(u)_j \in C_{\varepsilon}^n \wedge \abs{\vartheta^k_n(u)} \geqslant j] .
\]

Now, by definition, $\abs{\vartheta^k_n(u)(\omega)} = \abs{\vartheta^k(u_1)(\omega)} $ and $\vartheta^k_n(u)(\omega)_j = \vartheta^k(u)(\omega)_{[j,j+n-1]}$ for every realisation $\omega$. Thus,
\begin{equation}
\label{eq:Q-term}
Q_n^k(u,C^n_{\varepsilon}) \, = \, \frac{1}{\lambda^k} \sum_{j=1}^{\mu^k} \mathbb{P}[\vartheta^k(u)_{[j,j+n-1]} \in C_{\varepsilon}^n \wedge \abs{\vartheta^k(u_1)} \geqslant j]  ,
\end{equation}
where $\mu = \max_{\omega \in \Omega} \max_{a \in \mc A} \abs{\vartheta(a)(\omega)}$ is an upper bound for the 1-step inflation factor. Note that, at this point, $k \in \mathbb{N}$ is still a free parameter. 

Let us go back to Eq.~\eqref{eq:Q-term}. For the bulk of the remainder, we fix a position $j \in \mathbb{N}_0$ and consider the random variable given by $x^k = \vartheta^k(u)_{[j,j+n-1]}$, under the additional condition that $|\vartheta^k(u_1)| \geqslant j$. This ensures that $x^k$ overlaps the inflation word $\vartheta^k(u_0)$. Note that, in general, $x^k$ depends only on a prefix of $u$, given by $u_{[1,N_k+1]}$, where we take $N_k \in \mathbb{N}_0$ to be minimal with this property. Clearly, $N_k$ depends on the chosen realisation of $x^k$. It is thus also a random variable and we may regard it as a stopping time. By the minimality of $N_k$, we obtain that
\[
\vartheta^k(u_{[2,N_k(\omega)]})(\omega) \, \triangleleft \, x^k(\omega) \, \triangleleft \, \vartheta^k(u_{[1,N_k(\omega)+1]})(\omega)
\] 
holds for every realisation $\omega$. Note that the `interior' of $x^k$ on the left hand side is a concatenation of inflation words
\[
\widetilde x^k \, := \, \vartheta^k(u_{[2,N_k]}) \, = \, \vartheta^k(u_2)\cdot \ldots \cdot \vartheta^k(u_{N_k}),
\]
suppressing the dependence on $\omega$ in our notation. Compare Figure~\ref{Fig:quantities_in_long_proof} for a graphical representation of those quantities.
\begin{figure}
\begin{tikzpicture}
\node[right] at (0,0) {$u_1 u_2 \cdots \cdots u_N u_{N+1} \cdots \cdots \cdots u_n \quad = u$}; 

\draw (0.2, -0.2) -- (0.2,-3); 
\draw (0.6,-0.2) -- (1.7,-3); 
\draw (2.6,-0.2) -- (4.9,-3) (3.4,-0.2) -- (6.45,-3);
\node[right] at (0,-3.5) {$| \cdots \cdots \cdot \cdot | \cdots \cdots\cdots \cdots \cdots \cdots | \cdots \cdots \cdot \cdot | \cdots \cdots \cdots \cdots \cdots \quad = \vartheta^k(u)$}; 
\draw[->] (1.2,-3.1) -- (1.2,-3.3);
\draw (1.2,-3.2) node[above] {$\scriptstyle{j}$};
\draw[->] (5.8,-3.1) -- (5.8,-3.3);
\draw (5.8,-3.2) node[above] {$\scriptstyle{j+n}$};

\draw (1.2, -4) -- node[below]{$x^k$} (5.8 ,-4);
\draw (1.2, -3.9) -- (1.2,-4.1)  (5.8,-3.9) -- (5.8,-4.1);

\draw (1.7, -4.7) -- node[below]{$\widetilde{x}^k$} (4.95 ,-4.7);
\draw (1.7, - 4.6) -- (1.7,-4.8) (4.95 , -4.6) -- (4.95, -4.8);
\end{tikzpicture}
\caption{Illustration of how the random words $x^k$, $N = N_k$ and $\widetilde{x}^k$ are built from $\vartheta^k(u)$, given a fixed length $n \in \mathbb{N}$, a number of inflation steps $k \in \mathbb{N}$ and an initial position $1 \leqslant j \leqslant |\vartheta^k(u_1)|$.}
\label{Fig:quantities_in_long_proof}
\end{figure}
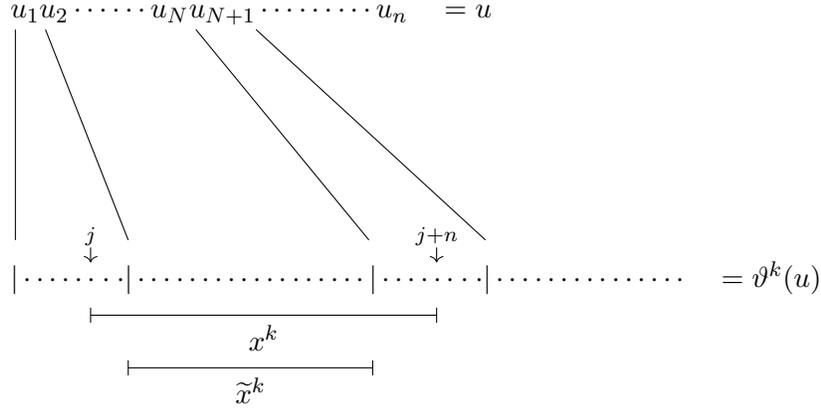
Note that the ratio of $|\widetilde{x}^k|$ and $n$ approaches $1$ as $n \rightarrow \infty$ since $|\vartheta^k(a)| \leqslant \mu^k$ for all $a \in \mc A$ (recall that $\mu$ denotes the maximal length of level-$1$ inflation words) and thus
\[
\frac{n - 2 \mu^k}{n} \, \leqslant \, \frac{n - |\vartheta^k(u_0)| - |\vartheta^k(u_{N_k +1})|}{n} \, \leqslant \, \frac{|\widetilde{x}^k|}{n} \leqslant 1.
\]
Consider the event $x^k \in C_{\varepsilon}^n$ that is equivalent to 
\begin{equation}
\label{eq:frequency-deviation}
\biggl| \frac{\abs{x^k}_v}{n} - R_v^{(\ell)} \biggr|  
\, > \, \varepsilon.
\end{equation}
We wish to express this in terms of the exact inflation word $\widetilde{x}^k$. We can split up the occurrences of $v$ in $x^k$ into those that are contained in $\widetilde{x}^k$, those contained in the overlap of $x^k$ with the boundary inflation words $\vartheta^k(u_1)$ or $\vartheta^k(u_{N_t+1})$ and finally those $v$ that overlap both $\widetilde{x}^k$ and one of the boundary inflation words. 
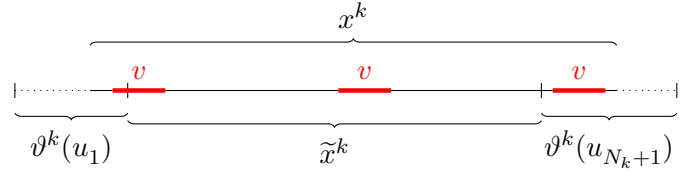
\begin{figure}
\begin{tikzpicture}
\draw[dotted] (0,0) -- (1,0); 
\draw (1,0) -- (1.3,0) (2,0)-- (4.3,0) (5,0) -- (7.15,0) (7.85,0)--(8,0) ; 
\draw[dotted] (8,0) -- (8.8,0);
\draw (0,-0.1) -- (0,0.1);
\draw (1.5,-0.1) -- (1.5,0.1);
\draw (7,-0.1) -- (7,0.1);
\draw (8.8,-0.1) -- (8.8,0.1);
\draw[decoration = {brace, mirror, raise=0.3cm}, decorate ] (0,0) -- node[anchor=north, yshift=-0.4cm]{$\vartheta^k(u^{}_1)$} (1.5,0);
\draw[decoration = {brace, mirror, raise=0.3cm}, decorate ] (7,0) -- node[anchor=north, yshift=-0.4cm]{$\vartheta^k(u^{}_{N_k+1})$} (8.8,0);
\draw[decoration = {brace, mirror, raise=0.4cm}, decorate ] (1.5,0) -- node[anchor=north, yshift=-0.5cm]{$\widetilde{x}^k$} (7,0);
\draw[decoration = {brace, raise=0.6cm}, decorate ] (1,0) -- node[anchor=south, yshift= 0.7cm]{$x^k$} (8,0);
\draw[line width=0.6mm,red] (1.3,0)--node[above]{$v$}(2.0,0);
\draw[line width=0.6mm,red] (4.3,0)--node[above]{$v$}(5.0,0);
\draw[line width=0.6mm,red] (7.15,0)--node[above]{$v$}(7.85,0);
\end{tikzpicture}
\caption{Possible positions of the word $v$ (highlighted in red) as a subword of $x^k$.}
\label{Fig:v-positions}
\end{figure}
For an illustration of the three qualitatively different positions of $v$ within $x^k$, compare Figure~\ref{Fig:v-positions}. 
Thus,
\begin{equation*}
|\widetilde{x}^k|_v \, \leqslant \, |x^k|_v
\, \leqslant \, |\widetilde{x}^k|_v + \abs{\vartheta^k(u^{}_1)}_v + \abs{\vartheta^k(u^{}_{N_k+1})}_v + 2\ell.
\end{equation*}
The last term originates from the fact that there are two boundaries of $\widetilde{x}^k$ with boundary inflation words each of which can contribute at most $\abs{v} = \ell$ occurrences of $v$ to the sum. For large enough $n$, we therefore conclude that
\[
\left| \frac{|x^k|_v}{n} - \frac{|\widetilde{x}^k|_v}{|\widetilde{x}^k|} \right| 
\, \leqslant \, \left| \frac{|x^k|_v}{n} - \frac{|\widetilde{x}^k|_v}{n} \right| + \left| \frac{|\widetilde{x}^k|_v}{n} - \frac{|\widetilde{x}^k|_v}{|\widetilde{x}^k|}  \right|
\, \leqslant \, \frac{\varepsilon}{2},
\]
such that Eq.~\eqref{eq:frequency-deviation} implies
\begin{equation}
\label{eq:rel-freq-deviation-2}
\left| \nu_v(\vartheta^k(u_{[2,N_k]})) - R_v^{(\ell)} \right|
= \left\vert \frac{|\widetilde{x}^k|_v}{|\widetilde{x}^k|} - R_v^{(\ell)} \right\vert 
\, > \, \frac{\varepsilon}{2}.
\end{equation}
Corresponding to $\frac{\varepsilon}{2}$, let us fix some $k \geqslant k_0$ as in Lemma~\ref{lemma:word-freq-prep}. Note that $N_k$ grows linearly with $n$ because
\[
N_k \, \leqslant \, n \, \leqslant \, |\vartheta^k(u_{[1,N_k+1]})| \, \leqslant \, \mu^k (N_k + 1),
\]
that is, there is some $C_1 > 0$ such that $ C_1 n \leqslant N_k(\omega) \leqslant n$ for large enough $n$, uniformly for all realisations $\omega \in \Omega$. An application of Lemma~\ref{lemma:word-freq-prep} yields
\begin{align*}
\mathbb{P}[\eqref{eq:frequency-deviation}]
& \, \leqslant \, \mathbb{P}[\eqref{eq:rel-freq-deviation-2}] 
\, = \, \sum_{N = C_1 n}^n \mathbb{P} \left[ \left| \nu_v(\vartheta^k(u_{[2,N]})) - R_v^{(\ell)}\right| > \frac{\varepsilon}{2}  \, \wedge \, N_k = N  \right]
\\& \, \leqslant \, \sum_{N = C_1 n}^n \me^{- CN} \, \leqslant \, n \me^{ - C C_1 n} \, \leqslant \, \me^{ - \widetilde{C} n} ,
\end{align*}
for large enough $n$ and appropriate choice of $\widetilde{C}>0$.
Since this bound holds irrespective of the choice of $j$, we have for any $j \in \{1,\ldots, \mu^k\}$,
\[
\mathbb{P}[\vartheta^k(u)_{[j,j+n-1]} \in C_{\varepsilon}^n \wedge \abs{\vartheta^k(u_1)} \geqslant j] 
\, \leqslant \, \mathbb{P}[\eqref{eq:frequency-deviation}]
\, \leqslant \, \me^{- \widetilde{C} n}.
\]
Comparing with Eq.~\eqref{eq:Q-term}, we find
\[
Q_n^k(u,C_{\varepsilon}^n) 
\, \leqslant \, \frac{\mu^k}{\lambda^k} \me^{-\widetilde{C} n}  \, \leqslant \, \me^{-C' n},
\]
for $n$ larger than some $n_0$, by another adjustment of the constant. Since this is summable over $n$, we find, using Eq.~\eqref{Eq:sum_deviations},
\[
\sum_{n=1}^{\infty} \sum_{ x \in C_{\varepsilon}^n} R_x^{(n)} 
\leqslant \sum_{n=1}^{n_0} \sum_{ x \in C_{\varepsilon}^n} R_x^{(n)} + \sum_{n=n_0+1}^{\infty} \underbrace{\sum_{u \in \mc L_{\vartheta}^n} R_u^{(n)}}_{=1} \me^{- C' n} < \infty.
\]
This finishes the proof.
\end{proof}

Finally, we are ready to prove the main result of this paper.

\begin{theorem} \label{thm:MAIN}
Let\/ $\vartheta$ be an expansive, primitive random substitution. Then, the measure-theoretic dynamical system\/ $(\mathbb{X}_{\vartheta}, \mathcal{B}_{\vartheta}, \mu_{\vartheta})$ is ergodic. 
\end{theorem}
\begin{proof}
It suffices to show that the identity
\begin{equation} \label{eq:ergodic}
\lim_{N\to\infty} \frac{1}{N} \sum_{j=0}^{N-1} f(S^j x) 
\, = \, \int_{\mathbb{X}_{\vartheta}}f\, \text{d}\mu_{\vartheta}
\end{equation}
holds for every $f\in L^1(\mathbb{X}_{\vartheta}, \mathcal{B}_{\vartheta}, \mu_{\vartheta})$ and almost all $x\in \mathbb{X}_{\vartheta}$.

First, let $\mathcal{Z}\in\mathfrak{Z}_{0}(\mathbb{X}_{\vartheta})$ be a cylinder set and $f=1_{\mathcal{Z}}$ be the indicator function on $\mathcal{Z}$. More precisely, let $\mathcal{Z}=\mathcal{Z}_{[k,m]}(v)$, where $k\leqslant m\in\Z$, and $v\in\mathcal{L}_{\vartheta}^{m-k+1}$. Now, by definition of $\mu_{\vartheta}$ and Proposition \ref{prop:mu-as-existence-of-word-frequencies}, we obtain
\[
\int_{\mathbb{X}_{\vartheta}}f\ \text{d}\mu_{\vartheta}
  \, = \, \mu_{\vartheta}(\mathcal{Z}_{[k,m]}(v)) 
  \, = \, R_v^{(m-k+1)} 
  \, = \, \lim_{N\to\infty}  \frac{|x_{[k,m+N-1]}|_v}{N} 
  \, = \, \lim_{N\to\infty} \frac{1}{N} \sum_{j=0}^{N-1} f(S^j x) ,
\]
for almost all $x$. Hence, Eq.~\eqref{eq:ergodic} is true if $f$ is the indicator function on a cylinder set.

Second, let $\varGamma:=\left\{ \sum_{\mathcal{Z}\in S} a_{\mathcal{Z}}\, 1_{\mathcal{Z}}\ |\ S \subseteq \mathfrak{Z}_0(\mathbb{X}_{\vartheta}) \text{ is finite and } a_{\mathcal{Z}}\in\C\right\}$ be the set of simple functions on $(\mathbb{X}_{\vartheta}, \mathcal{B}_{\vartheta}, \mu_{\vartheta})$. By linearity, the validity of Eq.~\eqref{eq:ergodic} for indicator functions on cylinder sets extends to arbitrary functions in $\varGamma$.

Third, since $\varGamma$ is dense in $L^1(\mathbb{X}_{\vartheta}, \mathcal{B}_{\vartheta}, \mu_{\vartheta})$, the claim follows.
\end{proof}

\section{Unique, strict and intrinsic ergodicity}
\label{sec:unique-strict-intrinsic}

The last part of this paper is devoted to studying further properties of the ergodic measures $\mu_{\vartheta}$ that we constructed in the preceding section. It turns out that random substitutions are a general enough class of objects to allow for a variety of different behaviours. We illustrate this by a number of examples. First, let us show that, in general, the measures $\mu_{\vartheta}$ are not uniquely ergodic.

\begin{example}
\label{ex:random-fib}
Consider the \emph{random period doubling substitution}
\[
\sigma:\ a\mapsto 
\begin{cases}
ab. & \text{ with probability } p \\
ba, & \text{ with probability } q:=1-p
\end{cases}, \quad \ b\mapsto aa,
\]
for some $0<p<1$. A short computation gives
\[
M_{\sigma}= 
\begin{pmatrix}
1 & 2 \\
1 & 0
\end{pmatrix}  \quad \text{ and } \quad M^{}_{\sigma^{}_2} =
\begin{pmatrix}
pq & q & 1+p & 2 \\
1-pq & p & q & 0 \\
1-pq & 1 & 0 & 0 \\
pq & 0 & 0  & 0 
\end{pmatrix}.
\]
Hence, we obtain $R^{(2)}=\left(\frac{2}{3(p^2-p+2)},\frac{2(1-p+p^2)}{3(p^2-p+2)},\frac{2(1-p+p^2)}{3(p^2-p+2)},\frac{p-p^2}{3(p^2-p+2)}\right)^{\intercal}$. Therefore,
\[
\int_{\mathbb{X}_{\sigma}} 1_{\mathcal{Z}_{0}(bb)}\ \text{d}\mu_{\sigma} \, = \, \mu_{\sigma}(\mathcal{Z}_{0}(bb)) \, = \, R_{bb}^{(2)} \, = \, \frac{p-p^2}{3(p^2-p+2)} \, > \, 0.
\]
On the other hand, we have
\[
\frac{1}{N} \sum_{k=0}^{N-1} 1_{\mathcal{Z}_{0}(bb)}(S^kx) \, = \, 0
\]
for all $N\in\N$ if $x$ is an element of the hull of the deterministic period doubling substitution (which is a subset of the random hull). Consequently, Eq.~\eqref{eq:ergodic} does not hold for all $x\in\mathbb{X}_{\sigma}$ and thus $\mu_{\sigma}$ is not uniquely ergodic.
\end{example}

In fact, it is possible to characterise those random substitution subshifts which are uniquely ergodic --- compare \cite[Thm.~27(b) and Cor.~28]{rs} for a proof of the following result.

\begin{prop} \label{prop:uni_ergo}
Let\/ $(\mathbb{X}_{\vartheta},S)$ be a primitive random substitution subshift. 
\begin{enumerate}
\item[(a)] $(\mathbb{X}_{\vartheta},S)$ is uniquely ergodic if and only if the right PF eigenvectors\/ $R^{(\ell)}$ are independent of the probability vectors\/ $\boldsymbol{p}_i$,
for every\/ $\ell$ and every\/ $i$.
\item[(b)] If\/ $(\mathbb{X}_{\vartheta},S)$ is uniquely ergodic, it is also strictly ergodic. \qed
\end{enumerate}  
\end{prop}

Most interesting examples of random substitutions however, give rise to highly non-minimal subshifts $\mathbb{X}$; see \cite{rs}. In that case, we can find an uncountable family of ergodic frequency measures with full support on $\mathbb{X}$. Let us expand a bit on this point. 

Recall that for a random substitution rule $P^{\vartheta}$, the subshift $\mathbb X_{\vartheta}$ depends only on the family $(\supp P^{\vartheta}_a)_{a \in \mc A}$. For each $a \in \mc A$, the values of $P^{\vartheta}_a$ on its support can therefore be regarded as free parameters with the additional constraint of forming a probability vector $\boldsymbol{p}_a$ on $\supp P^{\vartheta}_a$ with non-zero entries. Given any $\ell \in \mathbb{N}$, the entries of $M_{\vartheta_\ell}$ are polynomials in the entries of $\boldsymbol{p}_a$ for $a \in \mc A$ by construction and thus $R^{(\ell)}$ depends continuously on those parameters, compare $R^{(2)}$ in Example~\ref{ex:random-fib}. Assuming $R^{(\ell)}$ is not constant in all of the $\boldsymbol{p}_a$, it thus takes a continuum of values giving rise to a continuum of different ergodic frequency measures on $\mathbb{X}_{\vartheta}$.
We summarise this observation as follows.
\begin{coro}
Let\/ $\mathbb{X}$ be a random substitution subshift. For any pair\/ $\vartheta,\vartheta'$ of expanding primitive random substitution generating\/ $\mathbb{X}=\mathbb{X}_{\vartheta} = \mathbb{X}_{\vartheta'}$, the frequency measures\/ $\mu_{\vartheta}$ and\/ $\mu_{\vartheta'}$ are either equal or mutually singular. If\/ $(\mathbb{X},S)$ is not uniquely ergodic, there is an uncountable set of random substitutions\/ $\Theta$ such that\/ $\mu_{\vartheta} \perp \mu_{\vartheta'}$ for all\/ $\vartheta, \vartheta' \in \Theta$ with\/ $\vartheta \neq \vartheta' $. \qed
\end{coro}

Next, let us focus on the question whether or not a primitive random substitution subshift can be intrinsically ergodic, that is whether it allows a unique measure which maximises the entropy. We know that there is at least one measure of maximal entropy, see \cite[Sec. (17.15), Cor. 2]{dgs}. Consider the following two examples. 

\begin{example}
Let\/ $\mathbb{X}$ be a topologically transitive shift of finite type. It was shown in \cite{grs} that\/ $\mathbb{X}$ can be realised as a primitive RS-subshift. Since every topologically transitive shift of finite type is intrinsically ergodic \cite{parry}, there are intrinsically ergodic RS-subshifts. At present, it remains open whether the measure of maximal entropy on\/ $\mathbb{X}$ can be realised as a frequency measure\/ $\mu_{\vartheta}$ arising from some random substitution\/ $\vartheta$.
\end{example}

\begin{example}
The Dyck shift is a coded shift, which is not sofic. Moreover, it has two measures of maximal entropy \cite{krieger}. Hence, it is not intrinsically ergodic. The Dyck shift can be realised as a primitive RS-subshift via the random substitution
\[
(\, \mapsto \bigl\{((),\, ([\,],\, (\ \bigr\}, \quad \ \ [\, \mapsto \bigl\{ [(),\, [[\,],\, [\ \bigr\}, \quad )\, \mapsto \bigl\{()),\, [\,]),\, ) \bigr\},\quad ]\, \mapsto \bigl\{[\,]],\, ()],\, ] \bigr\},
\]
where the assignment of (non-degnerate) probability vectors is arbitrary. Therefore, not every RS-subshift is intrinsically ergodic.
\end{example}

To get a better feeling, let us compute the metric entropy for a specific example. To do so, we need the following result, which follows by an application of standard techniques from linear algebra; compare \cite[Thm.~1.1, Cor.~1]{seneta}.

\begin{lemma}  \label{lem:eigenvalue}
Let\/ $M\in\operatorname{Mat}(d,\R)$ be a primitive matrix such that\/ $M_{ij}=m_i\in \R_{\ge0}$ for all\/ $i,j\in\{1,\ldots,d\}$. Then,\/ $\lambda=\sum_{i=1}^dm_i$ is the PF eigenvalue of\/ $M$, and the corresponding (normalised) right PF eigenvector is given by 
\[
R \, = \, \frac{1}{\lambda}\, (m_1,\ldots,m_d)^{\intercal}.
\] \qed 
\end{lemma}

\begin{prop}
Consider the random substitution
\[
\zeta:\ a\mapsto 
\begin{cases}
ab, & \text{ with probability } p  \\
ba, & \text{ with probability } 1-p
\end{cases}, \quad \  b\mapsto 
\begin{cases}
ab, & \text{ with probability } p  \\
ba, & \text{ with probability } 1-p
\end{cases} ,
\]
for\/ $0<p<1$. If\/ $\mu$ denotes the corresponding frequency measure, the metric entropy\/ $h_{\mu}$ is given by
\[
h_{\mu} \,  = \, -\frac{1}{2}\, \Big( p\log(p) + (1-p)\log(1-p) \Big).
\]
\end{prop}
\begin{proof}
By definition of the metric entropy and the frequency measures, we obtain from Lemma~\ref{lem:eigenvalue}
\begin{equation}  \label{eq:met_ent}
\begin{split}
h_{\mu}
   & \, = \, \lim_{n\to\infty} -\frac{1}{n} \sum_{w\in \mathcal{L}_{\zeta}^n} \mu([w])\cdot 
       \log( \mu([w]))   
     \, = \,  \lim_{n\to\infty} -\frac{1}{n} \sum_{w\in \mathcal{L}_{\zeta}^n} R_{w}^{(n)}\cdot 
       \log \left(R_{w}^{(n)}\right)   \\
   & \, = \,  \lim_{n\to\infty} -\frac{1}{n} \sum_{i=1}^{|\mathcal{L}_{\zeta}^n|} \frac{1}{\lambda}
       m_i^{(n)}\cdot \log \left(\frac{1}{\lambda}m_i^{(n)} \right)   
     \, = \,  \frac{1}{\lambda} \cdot \lim_{n\to\infty} -\frac{1}{n} 
       \sum_{i=1}^{|\mathcal{L}_{\zeta}^n|} m_i^{(n)}\cdot  \log \left(m_i^{(n)}\right),
\end{split}
\end{equation}
where $m_i^{(n)} = \mathbb{E} |\zeta_n(u)|_{u_i}$, for $u,u_i\in \mathcal{L}_{\zeta}^n$, is independent of $u$. From now on, we consider the subsequence $(m_i^{(2n)})_{n\in\N}$. Let $u=u_1\ldots u_{2n}\in \mathcal{L}_{\zeta}^{2n}$. Then, for all $\omega \in \Omega$, $\zeta(u)(\omega)$ is of the form
\[
\zeta(u)(\omega) \, = \, \mathop{\odot}\limits_{k=1}^{2n} v^{k},\quad\ v^{k}\in\{ab,ba\},
\]
where $\odot$ denotes concatenation of words from left to right.
Therefore, 
\[
\zeta_{2n}(u)(\omega) \, = \, (v^{1}\ldots v^{n})\, (v^{1}_2 v^{2}\ldots v^{n}v^{n+1}_1).
\]
This implies
\begin{equation}
\label{eq:7-ex-splitting}
\mathbb{E}\big[ |\zeta_{2n}(u)|_{u_i}\big] 
    \, = \, \sum_{k=1}^2 \mathbb{E} \Big[\delta_{u_i}\big( \zeta(u)_{[k,k+2n-1]}\big)\Big]  \\
     \, = \, \sum_{k=1}^2 \mathbb{P}[\zeta(u)_{[k,k+2n-1]}=u_i].
\end{equation}
Note that, for a specific choice of the $v^{k}$, $k \in \{1,\ldots,n\}$,  
\[
\mathbb{P}\left[\zeta(u)_{[1,2n]} = \mathop{\odot}\limits_{k=1}^{n} v^{k}\right]  \, = \, p^{|\{k |v^{k}=ab\}|}\cdot(1-p)^{|\{k|v^{k}=ba\}|} .
\]
Since the suffix and prefix $v_2^{1}$ and $v_1^{n+1}$ determine the inflation words $v^{1}$ and $v^{n+1}$ uniquely, we also have that
\[
\mathbb{P}\left[\zeta(u)_{[2,2n+1]} = v_2^{1} \left( \mathop{\odot}\limits_{k=2}^{n} v^{k} \right) v_1^{n+1} \right] \, = \, p^{|\{k|v^{k}=ab\}|}\cdot(1-p)^{|\{k|v^{k}=ba\}|} ,
\]
where $k$ ranges from $1$ to $n+1$ on the right hand side. For $u_i$ to be legal, it must either be of the form $u_i=v^{1} \cdots v^{n}$ or $u_i = v_2^{1} v^{2} \cdots v^{n} v_1^{n+1}$, determining the words $v^{k} \in \{ab,ba\}$ uniquely in either case. The only words that can be written in both forms are $u_a = abab\ldots ab$ and $u_b = baba\ldots ba$. Note that there are exactly $\binom{n}{j}$ different words of the form $u=v^{1} \cdots v^{n}$ with $| \{ k = 1,\ldots,n | v^{k} = ab\}| = j$ for every $j \in \{0,\ldots,n\}$ and $\binom{n+1}{j}$ pairwise different words of the form $u_i = v_2^{1} v^{2} \cdots v^{n} v_1^{n+1}$ with $| \{ k = 0,\ldots,n | v^{k} = ab\}| = j$ for all $j \in \{0,\ldots,n+1\}$.
Together with Eq.~\eqref{eq:7-ex-splitting}, we obtain
\[
\begin{split}
\sum_{i=1}^{|\mathcal{L}_{\zeta}^{2n}|} m_i^{(2n)}\cdot  \log(m_i^{(2n)}) 
    & \, = \, \sum_{j=0}^n \binom{n}{j} p^j(1-p)^{n-j} \log\Big(p^j(1-p)^{n-j}\Big)  \\
    &\phantom{===}+ \sum_{j=0}^{n+1} \binom{n+1}{j} p^j(1-p)^{n+1-j} \log\Big(p^j(1-p)^{n+1-j}\Big)  \\
    &\phantom{===}+ \text{Err}_{2n}     \\
    & \, = \, \log(p)(2n+1)p + \log(1-p)(2n+1)(1-p) + \text{Err}_{2n},
\end{split}
\]
where the error term $\text{Err}_{2n} = \text{Err}_{2n}^a + \text{Err}_{2n}^b$, with
\begin{align*}
\text{Err}^a_{2n} & \, = \, \left( p^n + (1-p)^{n+1} \right) \log \left( p^n + (1-p)^{n+1} \right) - p^n \log\left(p^n\right) - (1-p)^{n+1} \log \left( (1-p)^{n+1} \right)
\\ \text{Err}^b_{2n} & \, = \, \left( p^{n+1} + (1-p)^{n} \right) \log \left( p^{n+1} + (1-p)^{n} \right) - p^{n+1} \log\left(p^{n+1}\right) - (1-p)^{n} \log \left( (1-p)^{n} \right)
\end{align*}
accounts for the fact that there are two words $u_a$ and $u_b$ for which both summands in Eq.~\eqref{eq:7-ex-splitting} are non-zero.
Clearly, we have $\frac{1}{2n}\text{Err}_{2n} \to 0$ for $n\to\infty$. Hence, we conclude
\[
\begin{split}
h_{\mu} 
    & \, = \, -\frac{1}{\lambda}\cdot \lim_{n\to\infty} \frac{1}{2n} \Big(\log(p)(2n+1)p + \log(1-p)(2n+1)(1-p) 
       + \text{Err}_{2n} \Big) \\
    & \, = \, -\frac{1}{2}\, \big(p\log(p)+(1-p)\log(1-p)\big),
\end{split}
\]
since $\lambda=2$.
\end{proof}

\begin{coro}
\label{coro:intrinsic-ergodic}
Let\/ $\zeta$ be as in the previous proposition. Then, the RS-subshift\/ $\mathbb{X}_{\zeta}$ is intrinsically ergodic, and the frequency measure which corresponds to\/ $p=\frac{1}{2}$ is the measure of maximal entropy.
\end{coro}
\begin{proof}
\begin{figure}
\centering
\begin{tikzpicture} 
  \SetGraphUnit{3}
  \Vertex{2}
  \WE(2){1}
  \EA(2){3}
  \Edge[label = $a$, color = red](1)(2)
  \Edge[label = $a$, color = red](2)(3)
  \Edge[label = $b$, color = blue](3)(2)
  \Edge[label = $b$, color = blue](2)(1)
  \tikzset{EdgeStyle/.append style = {bend left = 50}}
\end{tikzpicture}
\caption{The right-resolving graph representation of $X_\zeta$ as a sofic shift.}
\label{fig:sofic}
\end{figure}
It is not difficult to see that the subshift generated by the graph in Figure~\ref{fig:sofic} coincides with $\mathbb{X}_{\zeta}$. Therefore, it is a sofic shift and thus intrinsically ergodic.

The topological entropy of $\mathbb{X}_{\zeta}$ is given by $h(\mathbb{X}_{\zeta}) = \frac{1}{2}\log(2)$ which follows from simple combinatorics on $\mathbb{X}_{\zeta}$; alternatively, see \cite{gohlke} for a criterion that allows to read off $h(\mathbb{X}_{\zeta})$ immediately from the form of the random substitution. Now, if $p=\frac{1}{2}$, the previous proposition implies that $h_{\mu}= \frac{1}{2}\log(2)$, which finishes the proof.
\end{proof}

Consequently, if $\boldsymbol{p}=(p,1-p)^{\intercal}$ is the uniform distribution, we obtain the measure of maximal entropy. This result also holds for a larger class of random substitutions. For this, consider the random substitution
\[
\vartheta: a_i \mapsto 
\begin{cases}
w^{1}, & \text{ with probability } p_1  \\
\ \ \vdots & \quad \quad \quad \quad \vdots  \\
w^{\ell}, & \text{ with probability } p_{\ell} 
\end{cases}  \quad \quad \text{ for all } a_i\in\mathcal{A},
\]
where $\ell\in\N$, $\boldsymbol{p}=(p_1,\ldots,p_{\ell})^{\intercal}$ is a probability vector, and $w^{1},\ldots, w^{\ell}$ are legal words of length $N\in\N$, such that $\vartheta$ is primitive and $|w^{j_1}|_{a_i}=|w^{j_2}|_{a_i}$ for all $j_1,j_2\in\{1,\ldots,\ell\}$ and for all $a_i \in \mc A$.

\begin{prop}  \label{thm:max_ent}
Let\/ $\vartheta$ be as above. Then,\/ $\mathbb{X}_{\vartheta}$ is intrinsically ergodic. Moreover, the frequency measure corresponding to the vector\/ $\boldsymbol{p}=(\frac{1}{\ell},\ldots,\frac{1}{\ell})^{\intercal}$ is the measure of maximal entropy.
\end{prop}
\begin{proof}
The subshift $\mathbb{X}_{\vartheta}$ is a sofic shift, and therefore it is intrinsically ergodic; compare the previous proposition. 

By Eq.~\eqref{eq:met_ent}, the metric entropy can be computed via 
\[
h_{\mu} \, = \, \lim_{n\to\infty} -\frac{1}{n} \sum_{i=1}^{|\mathcal{L}_{\vartheta}^n|} \frac{1}{\lambda} m_i^{(n)} \cdot \log\left(\frac{1}{\lambda}m_i^{(n)}\right) 
 \, = \, \lim_{n\to\infty} -\frac{1}{n} \sum_{i=1}^{|\mathcal{L}_{\vartheta}^n|} \frac{1}{N} m_i^{(n)}\cdot \log \left(\frac{1}{N}m_i^{(n)}\right), 
\]
since $\lambda=N$ is the PF eigenvalue of $M_{\vartheta}$. Let us consider the subsequence $(Nn)_{n\in\N}$ of $(n)_{n\in\N}$. Let $u=u_1\ldots u_{Nn}\in \mathcal{L}_{\vartheta}^{Nn}$. Then, we have for all $\omega \in \Omega$
\[
\vartheta(u)(\omega) \, = \, \mathop{\odot}\limits_{j=1}^{Nn} v^{j}, \quad \quad v^{j} \in \{w^{1},\ldots,w^{\ell}\}
\]
by the inflation-word structure and thus
\[
\vartheta_{Nn}(u)(\omega) \, = \, (v^{1}\ldots v^{n})\, (v^{1}_{[2,N]}v^{2}\ldots v^{n}v_1^{n+1})\ldots (v_{N}^{1}v^{2}\ldots v^{n}v^{n+1}_{[1,N-1]}).
\]
due to the fact that all $v^{j}$ have the same length $N$.
Analogously to Eq.~\eqref{eq:7-ex-splitting}, we obtain
\[
m_i^{(Nn)} \, = \, \sum_{k = 1}^{N} \mathbb{P} \left[\vartheta(u)_{[k,k+Nn-1]} = u_i \right].
\]
Since every inflation word $w^{j}$, $j = 1,\ldots,\ell$ appears with the same probability $\frac{1}{\ell}$, we find that, for all $1 \leqslant k \leqslant N$ and $u_i \in \mc L^n_{\vartheta}$, the term $\mathbb{P} \left[ \vartheta(u)_{[k,k+Nn-1]} = u_i \right]$ is either $0$ or
\[
\ell^{-2} \ell^{-(n-1)} \, \leqslant \, \mathbb{P} \left[\vartheta(u)_{[k,k+Nn-1]} = u_i \right] \leqslant \ell^{-(n-1)}.
\]
Also, since $u_i$ is legal, at least one of the terms needs to be non-zero. This implies
\begin{equation} \label{eq:m_in}
\frac{1}{N} \ell^{-2} \ell^{-(n-1)} \, \leqslant \, \frac{1}{N} m_i^{(Nn)} \, \leqslant \, \ell^{-(n-1)}
\end{equation}
and 
\[
1 \, = \, \sum_{i=1}^{|\mathcal{L}_{\vartheta}^{Nn}|} \frac{1}{N} m_i^{(Nn)} 
\, \in \, \Bigl[|\mathcal{L}_{\vartheta}^{Nn}| \frac{1}{N}  \ell^{-2}\ell^{-(n-1)}, |\mathcal{L}_{\vartheta}^{Nn}| \ell^{-(n-1)}\Bigr], 
\]
which leads to
\[
\frac{1}{N}  \ell^{-2}\ell^{-(n-1)} \, \leqslant \, \frac{1}{|\mathcal{L}_{\vartheta}^{Nn}|} \, \leqslant \, \ell^{-(n-1)}.
\]
This together with Eq.~\eqref{eq:m_in} gives
\[
c_1 \frac{1}{|\mathcal{L}_{\vartheta}^{Nn}|} \, \leqslant \, \frac{1}{N} m_i^{(Nn)} \, \leqslant \, c_2 \frac{1}{|\mathcal{L}_{\vartheta}^{Nn}|}, 
\]
where $c_1:=\frac{1}{N\ell^2}$ and $c_2:=N\ell^2$. That is, the entries of the right eigenvector $R^{(Nn)}$ can differ from the uniform distribution only by a non-zero factor, which can be chosen \emph{independent} of $n$. This is enough to conclude that the metric entropy coincides with the topological entropy. More precisely, we have
\[
\begin{split}
h_{\mu}
     &\, = \,\lim_{n\to\infty} -\frac{1}{Nn} \sum_{i=1}^{|\mathcal{L}_{\vartheta}^{Nn}|} \frac{1}{N} m_i^{(Nn)}
         \log\left(\frac{1}{N} m_i^{(Nn)}\right)       \\
     & \, \geqslant \, \lim_{n\to\infty} -\frac{1}{Nn} \sum_{i=1}^{|\mathcal{L}_{\vartheta}^{Nn}|} 
       \frac{1}{N} m_i^{(Nn)}\log\left(c_1 \frac{1}{|\mathcal{L}_{\vartheta}^{Nn}|}\right)  \\
     &\, = \, \lim_{n\to\infty} -\frac{1}{Nn}\log\left(\frac{1}{|\mathcal{L}_{\vartheta}^{Nn}|}\right)  \, = \, h(\mathbb{X}_{\vartheta}),
\end{split} 
\]
i.e. $h_{\mu}\geqslant h(\mathbb{X}_{\vartheta})$. Since one always has $h_{\mu}\leqslant h(\mathbb{X}_{\vartheta})$, the claim follows. 
\end{proof}

Similarly as in Corollary~\ref{coro:intrinsic-ergodic}, this allows us to conclude
\begin{coro}
Let\/ $\vartheta$ be as above, and denote by\/ $\mu$ the measure of maximal entropy, i.e. the frequency measure that corresponds to the vector\/ $\boldsymbol{p}=(\frac{1}{\ell},\ldots,\frac{1}{\ell})^{\intercal}$. Then,
\[
h_{\mu} \, = \, \frac{1}{N} \, \log(\ell).
\]
\end{coro}

\begin{remark}
Proposition~\ref{thm:max_ent} gives an alternative characterisation of the measure of maximal entropy for a certain class of RS-subshifts, which are sofic shifts. Numerical computations suggest that this result can be extended to include non-sofic subshifts, such as the RS-subshift which arises from the random Fibonacci substitution 
\[
\vartheta_{\text{Fib}}:\ a\mapsto 
\begin{cases}
ba, & \text{ with probability } p \\
ab, &  \text{ with probability } 1-p
\end{cases},  \quad b\mapsto a.
\]
The difficulty in this situation is the computation of the normalised right PF eigenvectors\/ $R^{(n)}$ for large\/ $n$.
\end{remark}

\section{Appendix}

In this section we will make explicit a possible choice of the probability spaces involved in the construction of a random substitution $\vartheta$ and the corresponding Markov chain. Let us define  
$
\mc I_0 = \{ \langle a p \rangle \Mid a \in \mc A, p \in \mc \N \}. 
$
We call an element $I = \langle a p \rangle \in \mc I_0$ an \emph{individual} of type $a$ and position $p$. To each individual $I$ we attach a copy of $\mc A^+$ and define the product space $\Omega_2 = \bigtimes_{I \in \mc I_0} (\mc A^+)_I$. We equip $\mc A^+$ with the discrete topology, $\Omega$ with the corresponding product topology and define $\mc F_2$ to be the sigma-algebra of Borel sets. If $I = \langle ap \rangle$, we define $P_I = P^{\vartheta}_a$ and from this, we construct the product measure $\mathbb P_2 = \bigotimes_{I \in \mc I_0} P_I$. This yields a probability space $(\Omega_2, \mc F_2, \mathbb P_2)$. Now, suppose $\mc U$ is a random word on $(\Omega_1, \mc F_1, \mathbb P_1)$ and define $(\Omega, \mc F, \mathbb P) = (\Omega_1 \times \Omega_2, \mc F_1 \otimes \mc F_2, \mathbb P_1 \otimes \mathbb P_2)$. For a given $ \omega = (\omega_1, \omega_2) \in \Omega$, let $u=u_1 \cdots u_m = \mc U(\omega) = \mc U(\omega_1)$. From this, we define a subset of \emph{selected individuals} $\mc I_0^{\omega} = \{ \langle u_1 \, 1 \rangle, \ldots, \langle u_m \, m \rangle \}$. For each $j \in \{1,\ldots,m\}$, $\omega_{\langle u_j \, j\rangle}$ is a component of $\omega_2 = (\omega_I)_{I \in \mc I_0}$. Finally,
\[
\vartheta(\mc U)(\omega) := \omega_{\langle u^{}_1 \, 1\rangle} \cdots \omega_{\langle u^{}_m \, m \rangle} \in \mc A^+.
\] 
\begin{lemma}
\label{Lemma:App-conditional-prob}
The random words $\mc U$ and $\vartheta(\mc U)$ on $(\Omega, \mc F, \mathbb P)$ as defined above satisfy
\[
\mathbb P [\vartheta(\mc U) = v \mid \mc U = u ] = P(u,v),
\]
for all $u, v \in \mc A^+$.
\end{lemma}

\begin{proof}
Assume $\omega \in \Omega$ such that $\mc U(\omega) = u = u_1 \cdots u_m$. Then, $\vartheta(\mc U)(\omega) = \omega_{\langle u^{}_1 \, 1\rangle} \cdots \omega_{\langle u^{}_m \, m \rangle}$, by definition. Since $\mathbb P$ is a product measure, we obtain
\begin{align*}
\mathbb P [\vartheta(\mc U) = v \mid \mc U = u ]
& = \mathbb P [\omega_{\langle u^{}_1 \, 1\rangle} \cdots \omega_{\langle u^{}_m \, m \rangle} = v] = \sum_{ v^1 \cdots v^m = v} \mathbb P[\omega_{\langle u_1 \, 1 \rangle} = v^1] \cdots \mathbb P[\omega_{\langle u_m \, m \rangle} = v^m]
\\ & = \sum_{ v^1 \cdots v^m = v} P^{\vartheta}_{u_1}(v^1) \cdots P^{\vartheta}_{u_m}(v^m) 
= P(u,v),
\end{align*}
where we have used $\sum_{v^1 \cdots v^m = v}$ as a shorthand notation for $\sum_{\substack{v^1, \ldots, v^m \in \mc A^+ \\ v^1 \cdots v^m = v}}$.
\end{proof}

In order to iterate $\vartheta$ we might in principle always take the product with the same probability space $(\Omega_2, \mc F_2, \mathbb P_2)$, outlined above. However, it will prove convenient to equip the individuals in subsequent `generations' with a bit more structure.
The following construction is along the same lines as for the multitype Galton--Watson process, compare \cite[Ch.~1.2]{mode}. We focus on the case that we start from a deterministic word $u = u_1 \cdots u_m \in \mc A^+$. We form a set of \emph{initial individuals} $\mc I_0 = \{ \langle u_1 \, 1 \rangle, \ldots, \langle u_m \, m \rangle \}$. 
An \emph{individual of generation $n$} is a tuple $I = \langle a_0 p_0 a_1 p_1 ... a_n p_n \rangle$ with $n \in \mathbb N$, $\langle a_0 p_0 \rangle \in \mc I_0$, $a_i \in \mc A$ and $p_i \in \mathbb{N}$ for all $1 \leqslant i \leqslant n$. We call $a(I) := a_n$ the type of the individual and $p_n$ its position. 
We say that $I = \langle a_0 p_0 \cdots a_k p_k \rangle$ is an ancestor of $J = \langle a'_0 p'_0 \cdots a'_{k+s} p'_{k+s} \rangle$ if $a_j = a'_j$ and $p_j = p'_j$ for all $1\leqslant j \leqslant k$. 
For example, we want to think of $\langle a_0 p_0 a_1 p_1 \rangle$ as an individual of type $a_1$ produced from $\langle a_0 p_0 \rangle$ and occurring at position $p_1$ in the word that $\langle a_0 p_0 \rangle$ is substituted with. 
The collection of all possible individuals forms a countable set
\[
\mc I = \bigcup_{n \geqslant 0} \mc I_n, \quad \mc I_n = \{ \langle a_0 p_0 \cdots a_n p_n \rangle \Mid \langle a_0 p_0 \rangle \in \mc I_0,  a_i \in \mc A, \, p_i \in \mathbb{N} \; \mbox{for all } 1 \leqslant i \leqslant n \}.
\]
From this, we define
$
\Omega_u = \bigtimes_{I \in \mc I} (\mc A^+)_I
$
with Borel sigma-algebra $\mc B_u$ and probability measure $\mathbb{P}_u = \bigotimes_{I \in \mc I} P_I$. Again, $P_I = P^{\vartheta}_{a}$, where $a$ is the type of $I$.
On finite cylinders of the form $B = \prod_{I} B_I$ where $B_I \neq \mc A^+$ only for $I$ in a finite subset $\mc J \subset \mc I$,  this measure is given by
\[
\mathbb{P} \biggl( \prod_{I} B_I \biggr) = \prod_{I \in \mc J} P_I (B_I).
\]
In the process that we want to model, only a finite subset of $\mc I_n$ is produced from the previous generation in each step. For $\omega = (\omega_I)_{I \in \mc I} \in \Omega_u$, the set of selected individuals $\mc I_n^{\omega} \subset \mc I_n$ for each generation $n$ is defined inductively. First, $\mc I_0^{\omega} \equiv \mc I_0$, reflecting the fact that we start from a deterministic word $u$. Suppose $I = \langle a_0 p_0 \cdots a_n p_n \rangle \in \mc I_n^{\omega}$ for some $n \in \N_0$ and $\omega_I = v_1 \cdots v_k$, with $k \in \mathbb N$. The set of all individuals produced by $I$ is given by
$
A_I = \{ \langle a_0 p_0 \cdots a_n p_n v_j j \rangle \Mid j \in \{1,\ldots, k\} \}
$ and we define $\mc I_{n+1}^{\omega} = \cup_{I \in \mc I_n^{\omega}} A_I$. 
\\ In the next step, we construct a word from each of the sets $\mc I_n^{\omega}$ that we call a \emph{level-$n$ inflation word} of the initial word $u$. For $n \in \mathbb N_0$, we define a strict total order relation $<$ on $\mc I_n^{\omega}$ as follows. We set $\langle a_0 p_0 \cdots a_n p_n \rangle < \langle a'_0 p'_0 \cdots a'_n p'_n \rangle$ if there is an $i \in \{0,\ldots,n\}$ such that $p_i < p'_i$ and $p_j = p'_j$ for all $j < i$. 
Suppose $I_n^1 < I_n^2 < \ldots < I_n^r$ is the ordering of the individuals in  $\mc I_n^{\omega}$. 
Then, for every $n \in \mathbb N_0$,
\[
\vartheta^{n+1}(u) \colon \Omega_u \to \mc A^+, \quad \vartheta^{n+1}(u)(\omega) : = \omega_{I_n^1} \cdots \omega_{I_n^r}.
\]
Note that we have suppressed the dependence of $I_n^j$ on $\omega$ for notational convenience. It is worth noticing that if $\vartheta^{n+1}(u)(\omega) = v_1 \cdots v_s$, then $s = \card{\mc I_{n+1}^{\omega}}$ and $v_j$ is the type of $I_{n+1}^j$ for $j \in \{1,\ldots,s\}$.

\begin{lemma}
The sequence $(\vartheta^n(u))_{n \in \mathbb N}$ forms a stationary Markov chain with transition kernel
\[
\mathbb P_u [\vartheta^{n+1}(u) = w | \vartheta^n(u) = v] = P(v,w),
\]
for all $n \in \mathbb{N}$ and $v,w \in \mc A^+$.
\end{lemma} 

The proof is analogous to the one presented for Lemma~\ref{Lemma:App-conditional-prob} and is left to the interested reader.

\begin{example}
\label{EX:RFib-p-space}
Consider the random Fibonacci substitution rule on $\mc A = \{a,b\}$, given by $P_b^{\vartheta} = \delta_a$ and $P_a^{\vartheta} = p \delta_{ab} + q \delta_{ba}$, with $p \in (0,1)$ and $q = 1-p$. Let us start from the initial word $u = a$, so $\mc I_0 = \{ \langle a \, 1 \rangle \}$. Suppose $\omega = (\omega_I)_{I \in \mc I} \in \Omega_a$, with $\omega_{\langle a \,1 \rangle} = ab$. Then, $\mc I_1^{\omega} = \{ \langle a 1 a 1\rangle , \langle a 1 b 2 \rangle \}$. These are ordered as $\langle a1 a1 \rangle < \langle a1 b2 \rangle$.
Suppose further $\omega_{\langle a1 a1 \rangle} = ab$ and $\omega_{\langle a1 b2 \rangle } = a$. Consequently, $\vartheta^2(a)(\omega) = aba$; compare Figure~\ref{Fig:omega-fib} for a graphical interpretation.
On the other hand $\omega' = (\omega'_I)_{I \in \mc I} \in \Omega$ with $\omega'_{\langle a 1 \rangle} = ba$, $\omega'_{\langle a1 b1 \rangle} = a$, $\omega'_{\langle a1 a2 \rangle} = ba$ yields again $\vartheta^2(a)(\omega') = aba$. Since these are the only ways to obtain the word $aba$ in two steps with non-vanishing probability, we obtain
\begin{align*}
\mathbb P_a [\vartheta^2(a) = aba ] &= \mathbb P_a [\omega_{\langle a 1 \rangle} = ab \wedge \omega_{\langle a1 a1 \rangle} = ab \wedge \omega_{\langle a1 b2 \rangle } = a] 
\\ &\quad + \mathbb P_a [\omega_{\langle a1 \rangle} = ba \wedge \omega_{\langle a1 b1 \rangle} = a \wedge \omega_{\langle a1 a2 \rangle } = ba] 
= p^2 + q^2.
\end{align*}
\end{example}

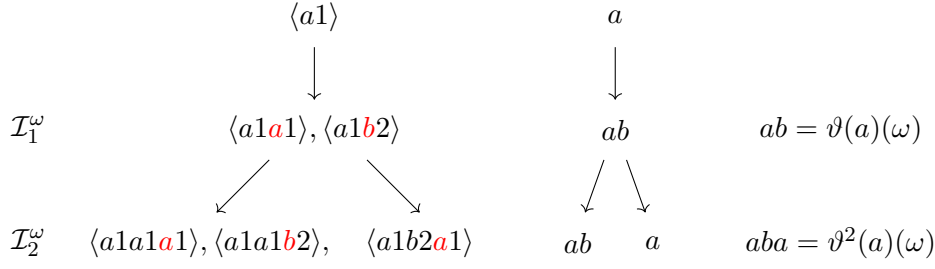
\begin{figure}
\label{Fig:omega-fib}
\begin{tikzpicture}
\node at (-3.8,1.5) {$\mc I_1^{\omega}$};
\node at (-3.8,0) {$\mc I_2^{\omega}$};
\node at (0,3) {$\langle a 1 \rangle$};
\node at (0,1.5) {$\langle a1 {\color{red}a} 1 \rangle, \langle  a1 {\color{red}b}2 \rangle$};
\node at (-1.4,0) {$\langle a1 a1 {\color{red}a}1 \rangle, \langle a1 a1 {\color{red}b}2 \rangle,$};
\node at (1.4,0) {$\langle a1 b2 {\color{red}a}1 \rangle $};
\draw[->] (0,2.6) -- (0,1.9);
\draw[->] (-0.6, 1.1) -- (-1.3 , 0.4);
\draw[->] (0.7, 1.1) -- (1.4 , 0.4);
\node at (4,3) {$a$};
\node at (4,1.5) {$ab$};
\node at (3.5,0) {$ab$};
\node at (4.5,0) {$a$};
\draw[->] (4,2.6) -- (4,1.9);
\draw[->] (3.85, 1.1) -- (3.6 , 0.4);
\draw[->] (4.15, 1.1) -- (4.4 , 0.4);
\node at (7,1.5) {$ab = \vartheta(a)(\omega)$};
\node at (7,0) {$aba = \vartheta^2(a)(\omega)$};
\end{tikzpicture}
\caption{Left hand side: Illustration of the selected individuals of the first two generations constructed from $\omega$ in Example~\ref{EX:RFib-p-space}. The types of the individuals are colored in red. The tree on the right displays the choices of $\omega_I$ for $I = \langle a1 \rangle $ (second line) and for $I \in \mc I_1^{\omega}$ (third line).}
\end{figure}


Given $u = u_1 \cdots u_m$, we will next discuss how to embed random variables $\vartheta^n(u_j)$ for $j \in \{1,\ldots,m\}$ into the probability space $(\Omega_u, \mc F_u, \mathbb P_u)$. Note that we can partition $\mc I = \cup_{j=1}^m \mc I(j)$ with $I \in \mc I(j)$ if and only if $I$ is of the form $I = \langle u_j j \cdots \rangle$. This leads to a decomposition $\Omega_u = \Omega_1 \times \cdots \times \Omega_m$, as well as $\mc B_u = \mc B_1 \otimes \cdots \otimes \mc B_m$ and $\mathbb P_u = \mathbb P_1 \otimes \cdots \otimes P_m$, where $(\Omega_j, \mc B_j, \mathbb P_j)$ is the probability space constructed from $\mc I(j)$ for all $j \in \{1,\ldots, m\}$. Note that $(\Omega_j, \mc B_j, \mathbb P_j)$ is just a copy of the probability space $(\Omega_{u_j}, \mc B_{u_j}, \mathbb P_{u_j})$. Thus, $\vartheta^n(u_j)$ is well-defined on $(\Omega_j, \mc B_j, \mathbb P_j)$ and fulfils the defining relations of the corresponding Markov chain. For $\omega = (\omega_1, \ldots, \omega_m) \in \Omega_u$, we define $\vartheta^n(u_j)(\omega) = \vartheta^n(u_j)(\omega_j)$ for all $j \in \{1,\ldots,m\}$. Because of the product structure of $(\Omega_u,\mc B_u,\mathbb P_u)$, the random variables $\vartheta^n(u_j)$ and $\vartheta^n(u_k)$ are independent for every $n \in \N$ and $j \neq k$. Finally note that for $\omega \in \Omega_u$, $\mc I_n^{\omega} = \cup_{j = 1}^m ( \mc I_n^{\omega} \cap \mc I(j) )$ and for $j < k$ all the elements in $\mc I_n^{\omega} \cap \mc I(j)$ are smaller than those in $\mc I_n^{\omega} \cap \mc I(k)$. Since $\vartheta^n(u_j)$ depends only on the individuals in $\mc I_{n-1}^{\omega} \cap \mc I(j)$, this implies
\[
\vartheta^n(u)(\omega) = (\vartheta^n(u_1) \cdots \vartheta^n(u_m) )(\omega)
\]
for all $\omega \in \Omega_u$ and $n \in \N$.

\section*{Acknowledgements} 

\noindent It is a pleasure to thank Michael Baake and Dan Rust for helpful discussions.
 
This work is supported by the German Research Foundation (DFG) via
the Collaborative Research Centre (CRC 1283) through the faculty of Mathematics, Bielefeld University. Also, PG was partially supported by the Research Centre of Mathematical Modelling (RCM$^2$) of Bielefeld University, while TS was also supported by NSERC via grant 03762-2014.

\end{document}